\documentclass[11 pt,a4 paper, oneside]{amsart}
\usepackage{latexsym}
\usepackage[top=2.5cm, bottom=3cm, left=2.5cm, right=2.5cm]{geometry}
\usepackage[utf8]{inputenc}

\usepackage[T1]{fontenc}

\usepackage[english]{babel}

\usepackage{bm}
\usepackage{amsthm}
\usepackage{amssymb}
\usepackage{amsmath}
\usepackage{amsrefs}
\usepackage{indentfirst}
\usepackage{graphicx}
\usepackage{mathrsfs}

\usepackage{enumitem}

\usepackage[hidelinks]{hyperref}
\hypersetup{
	colorlinks,
	linkcolor={blue},
	citecolor={blue},
	urlcolor={black}
}

\usepackage{caption}
\usepackage{changepage}
\usepackage{cleveref}

\usepackage{comment}

\usepackage{rotating}
\usepackage{float}
\usepackage{flafter} 
\usepackage{tabularx}

\usepackage{tikz} 
\usepackage{tikz-cd}
\usepackage{amsfonts}
\newtheorem{theorem}{Theorem}[section]
\newtheorem{lemma}[theorem]{Lemma}
\newtheorem{proposition}[theorem]{Proposition}
\newtheorem{corollary}[theorem]{Corollary}
\theoremstyle{definition}
\newtheorem{definition}[theorem]{Definition}
\newtheorem{remark}[theorem]{Remark}
\newtheorem{thmintro}{Theorem}

\title{Strong formality below middle degree implies strong formality}
\author{Lapo Rubini}
\address[Lapo Rubini]{Dipartimento Di Matematica E Informatica “Ulisse Dini”, Università Degli Studi di Firenze, viale Morgagni 67/a, 50134, Florence, Italy}
\email{lapo.rubini@unifi.it}

\newcommand{\C}{\mathbb{C}}

\newcommand{\CP}{\mathbb{CP}}
\newcommand{\ddbar}{\partial\bar{\partial}}

\thanks{The author has been supported by the PRIN 2022 project "Real and Complex Manifolds: Geometry and holomorphic dynamics" (code	2022AP8HZ9) and partially supported by GNSAGA-INdAM}

\begin{document}

\begin{abstract}
	We show that hyperplane sections of strongly formal manifolds inherit strong formality. In particular, this property holds for generalized complete intersections defined by positive line bundles with trivial first de Rham cohomology group. Furthermore, we establish the strong formality of compact K\"ahler manifolds with central cohomology of width $\frac{n}{2}-1$ and, more generally, of compact $\ddbar$-manifolds with no non-trivial multiplicative relations in cohomology below degree $n+2$. These results arise from the notion of $s$-strong formality, which we adapt from a work of Fernandez and Muñoz to the pluripotential setting. Specifically, we prove that a compact, connected complex manifold of dimension $n$ with trivial first de Rham cohomology group is strongly formal if and only if it is $(n-1)$-strongly formal.
\end{abstract}

\pagestyle{plain}
	\maketitle

\section*{Introduction}
In \cite{Sullivan75}, Sullivan introduced the classical notion of formality for differentiable manifolds: a manifold is formal if its cdga of forms is quasi-isomorphic to its de Rham cohomology algebra with zero differential. In the remarkable paper \cite{DeligneGriffiths1975}, Deligne, Griffiths, Morgan, and Sullivan proved, among other fundamental results, that all compact K\"ahler manifolds are formal in the sense of Sullivan. This is achieved by means of the $\ddbar$-Lemma, a property of independent interest for non-K\"ahler manifolds: namely a compact complex manifold satisfies the $\ddbar$-Lemma if every $d$-exact complex form that is both $\partial$ and $\bar{\partial}$-closed is also $\ddbar$-exact. While this property establishes a connection between complex geometry and Sullivan's formality, the latter is independent from the complex structure of the manifold. 

The introduction of Dolbeault formality in \cite{Neisendorfer78} represented an initial attempt to provide a complex analogue to Sullivan's formality. However, the more recent notion of strong formality defined in \cite{milivojevic2024bigradednotionsformalityaepplibottchernmassey} is symmetric in the two differentials $\partial$ and $\bar{\partial}$, providing a more comprehensive parallel for complex manifolds. A strongly formal manifold has its cbba of complex differential forms that is pluripotentially quasi-isomorphic to its Bott-Chern cohomology algebra with zero differentials. This condition implies, in particular, that the Hodge structure on the homotopy groups of the manifold is determined by the Hodge structure of its cohomology ring (see \cite[\S5.4]{placini2025strongformalitytorichomogeneous}, \cite[\S2.3.1]{Stelzig_2025}). An obstruction to strong formality is provided by the presence of non-zero ABC-Massey products (see \cite{milivojevic2024bigradednotionsformalityaepplibottchernmassey, Angella_2015, cesarino2025aepplibottchernmasseyproductsnonkahler, placini2024nontrivialmasseyproductscompact, sferruzza2024bottchernformalitymasseyproducts, SFERRUZZA2022104470, martínmerchán2025holomorphiclinkingnumbersabc, tardini2015geometricbottchernformalitydeformations}). In contrast to the differentiable case, compact K\"ahler manifolds and, more generally, compact $\ddbar$-manifolds, are not necessarily strongly formal (see \cite{placini2024nontrivialmasseyproductscompact, SFERRUZZA2022104470, martínmerchán2025holomorphiclinkingnumbersabc}).

Since the $\ddbar$-Lemma is a necessary condition for a compact manifold to be strongly formal, it is interesting to investigate which classes of compact $\ddbar$-manifolds are guaranteed to be strongly formal. Recent works \cite{placini2025strongformalitytorichomogeneous, RUBINI2025105586, sferruzza2024bottchernformalitymasseyproducts, sferruzza2025hermitiangeometricallyformalmanifolds, Stelzig_2025} have provided positive answer for several classes: complete smooth complex toric varieties; homogeneous compact K\"ahler manifolds; splitting-type solvmanifolds of the form ${\C}^n\ltimes{\C}^m/\Gamma$ satisfying the $\ddbar$-Lemma; compact K\"ahler solvmanifolds; compact geometrically Aeppli formal manifolds; compact K\"ahler manifolds of dimension $n\ge2$ with the Hodge diamond of complete intersection type.

This paper aims to provide some techniques for establishing the strong formality of further classes of $\ddbar$-manifolds. In \cite{fernandezmunoz}, Fernandez and Muñoz define the weaker notion of $s$-formality, providing that for connected, orientable and compact differentiable manifolds of dimension $2n$ or $2n-1$, formality is equivalent to $(n-1)$-formality.

As a first step, aiming to adapt the results in \cite{fernandezmunoz} to the context of strong formality, we establish the following characterization.

\begin{thmintro}[Theorem \ref{characterization}]\label{charintro}
	A minimal cbba $(\mathcal{M}, \partial, \bar{\partial})$ is strongly formal if and only if it satisfies the $\ddbar$-Lemma and there exists a bigraded vector space of generators $V$, $\mathcal{M}=\Lambda V$, which decomposes as $V=C\oplus N$ as bigraded vector spaces such that:
	\begin{enumerate}[label=(\roman*)]
		\item $d(C^{p,q})=0$ for all $p,q$;
		\item $d$ is injective on each $N^{p,q}$;
		\item each $\ddbar$-closed element in the ideal generated by $\bigoplus_{p,q}N^{p,q}$ is in $\partial\mathcal{M}+\bar{\partial}\mathcal{M}$;
		\item each $\partial$ and $\bar{\partial}$-closed element in the ideal generated by $\bigoplus_{p,q}N^{p,q}$ is in $\ddbar\mathcal{M}$.
	\end{enumerate}
\end{thmintro}

This result, analogous to the characterization of formality proved in \cite[Theorem 4.1]{DeligneGriffiths1975}, allows us to define an {\em $s$-strongly formal minimal cbba} (Definition \ref{sstrong}) by requiring the splitting of generators and the corresponding properties to hold up to degree $s$, alongside with an "up to degree $s$" $\ddbar$-Lemma. Following the approach in \cite{fernandezmunoz}, we provide the following result. Note that for connected and compact $\ddbar$-manifolds with trivial first de Rham cohomology group, the existence of a bigraded minimal model is guaranteed by \cite[Theorem 2.34]{Stelzig_2025}.

\begin{thmintro}[Theorem \ref{half}]\label{thmintros}
	Let $X$ be a compact, connected $\ddbar$-manifold with $H_{dR}^1(X)=0$ and $\dim_{\C}(X)=n$. Then $X$ is strongly formal if and only if $X$ is $(n-1)$-strongly formal.
\end{thmintro}

Thanks to this result, we establish the strong formality of several classes of compact complex manifolds. First, we consider compact $n$-dimensional K\"ahler manifolds with central cohomology of width $\frac{n}{2}-1$, namely those whose Hodge diamond (excluding the powers of the K\"ahler class) has non-zero entries only in degrees ranging from $\frac{n}{2}+1$ to $\frac{3n}{2}-1$.

\begin{thmintro}[Theorem \ref{centralcohomology}]\label{thmintros2}
	Any compact K\"ahler manifold $(X,\omega)$, $\dim_{\C} X=n$, with central cohomology of width $\frac{n}{2}-1$, is strongly formal. For $n=4m-1$, the same result holds also if $h^{m,m}(X)=2$.
\end{thmintro}

This result implies, in particular, that Fano threefolds with Picard group of rank $1$ or $2$ and Fano fourfolds with Picard group of rank $1$ are strongly formal.

The same techniques can be applied more in general to manifolds satisfying the $\ddbar$-Lemma and with no multiplicative relations in cohomology below degree $n+2$, beyond those imposed by the graded-commutativity. A concrete example is given by Clemens manifolds with vanishing second Betti number (see \cite{Friedman91, Chi24}). This leads to the following.

\begin{thmintro}[Theorem \ref{generalthm}]\label{thmintros3}
	A compact, connected and holomorphically simply connected $\ddbar$-manifold of complex dimension $n$ with no non-trivial multiplicative relations in cohomology below degree $n+2$ is strongly formal.
\end{thmintro}

Finally, Theorem \ref{thmintros} can be applied under the hypotheses of the Lefschetz hyperplane Theorem.

\begin{thmintro}[Theorem \ref{hypThm}]\label{thmintro4}
	Let $X$ be a compact, connected K\"ahler manifold of complex dimension $n+1\ge3$ with $H_{dR}^1(X)=0$. Let $Y\subseteq X$ be a smooth hypersurface such that the induced line bundle $\mathcal{O}(Y)$ is positive. If $X$ is $(n-1)$-strongly formal, then $Y$ is strongly formal.
\end{thmintro}

As a corollary, we obtain the strong formality of generalized complete intersection manifolds defined by positive line bundles with trivial first de Rham cohomology group.\\

This paper is organized as follows. In Section \ref{section1}, we review the necessary preliminaries, providing the fundamental definitions concerning commutative bigraded bidifferential algebras and strong formality. Section \ref{section2} is devoted to introducing the notion of $s$-strong formality and establishing algebraic results for minimal cbbas; in particular, we provide the proofs of Theorem \ref{charintro} and the algebraic counterparts of Theorems \ref{thmintros} and \ref{thmintro4}, namely Theorems \ref{halfAlg} and \ref{hypThmAlg}. In Section \ref{section3}, these results are applied to the setting of complex geometry. By virtue of Theorem \ref{thmintros}, we establish Theorem \ref{thmintros2} and its generalization, Theorem \ref{thmintros3}. These results are then applied to prove the strong formality of Fano threefolds with Picard group of rank $1$ or $2$, Fano fourfolds with Picard group of rank $1$ and Clemens manifolds. Finally, we provide a proof for Theorem \ref{thmintro4}, concluding with an application to generalized complete intersection manifolds defined by positive line bundles with trivial first de Rham cohomology group.\\

\noindent{\em Acknowledgments.} The author would like to express sincere gratitude to Daniele Angella and Jonas Stelzig for their constant support and for the valuable suggestions that made the completion of this paper possible. 

\section{Preliminaries and notation}\label{section1}

In this section, we recall the fundamental notions and properties of commutative bigraded bidifferential algebras and the concept of strong formality, as introduced in \cite{Stelzig_2025} and \cite{milivojevic2024bigradednotionsformalityaepplibottchernmassey} respectively. For any bigraded object $A$, we adopt the following notation:
\begin{itemize}
	\item $A^i:=\bigoplus_{p+q=i}A^{p,q}$ denotes the total degree $i$ elements;
	\item $A^{\le s}:=\bigoplus_{p+q\le s}A^{p,q}$ denotes the elements of total degree at most $s$;
	\item $A:=\bigoplus_{p,q}A^{p,q}$ denotes the entire bigraded space;
	\item for an element $\eta \in A$, the notation $|\eta|$ denotes either the total degree or the bidegree of $\eta$, depending on the context.
\end{itemize}

\subsection{Commutative bigraded bidifferential algebras.} A bicomplex (or double complex) is a bigraded $\C$-vector space $A=\bigoplus_{p,q}A^{p,q}$ equipped with an endomorphism $d$ of total degree $1$ that decomposes into components $d=\partial+\bar{\partial}$ of bidegrees $(1,0)$ and $(0,1)$ respectively, satisfying $d^2=0$. A bicomplex is degree-wise finite dimensional if $A^{p,q}$ is a finite dimensional complex vector space for all $p,q$. 

A commutative bigraded bidifferential algebra (cbba) over $\C$ is a bicomplex $(A,\partial,\bar{\partial})$ endowed with a product $\wedge:A\times A\to A$ of bidegree $(0,0)$ and a unit $\C\to A$ that makes $A$ into a unital graded-commutative algebra such that $\partial$ and $\bar{\partial}$ satisfy the Leibniz rule with respect to $\wedge$. 

An augmented cbba is a pair $(A,\varepsilon)$, where $A$ is a cbba and $\varepsilon:A\to\C$ is a cbba morphism (here $\C$ is a cbba concentrated in bidegree $(0,0)$). We denote the augmentation ideal by $A^+:=\ker\varepsilon$. A real structure on a cbba $A$ is a $\C$-linear isomorphism of cbbas $\sigma:A\simeq\bar{A}$, where $\bar{A}$ denotes the complex conjugate of the double complex $A$. A cbba $A$ is connected if it is concentrated in non-negative total degree and the unit $\C\to A^0$ is an isomorphism. A cbba $A$ is simply connected if it is connected and $A^1=0$.

The Dolbeault, Bott-Chern and Aeppli cohomologies of a cbba $(B, \partial,\bar{\partial})$ are defined as
\[
	H_{\bar{\partial}}^{\bullet,\bullet}(B):=\dfrac{\ker\bar{\partial}}{\text{Im}\bar{\partial}};\quad H_{BC}^{\bullet,\bullet}(B):=\dfrac{\ker\partial\cap\ker\bar{\partial}}{\text{Im}\ddbar};\quad H_A^{\bullet,\bullet}(B):=\dfrac{\ker\ddbar}{\text{Im}\partial+\text{Im}\bar{\partial}}.
\]

A cbba $(B, \partial,\bar{\partial})$ is said to satisfy the $\ddbar$-Lemma if every pure degree $d$-exact element in $B$ which is both $\partial$ and $\bar{\partial}$-closed is also $\ddbar$-exact. This condition is satisfied on the cbba of complex differential forms of a compact K\"ahler manifold by \cite[Lemma 5.11]{DeligneGriffiths1975}. In the same article, the authors provide several equivalent conditions to this property (cf. \cite[Lemma 5.15, Lemma 5.17]{DeligneGriffiths1975}).

A cbba morphism is a bigraded algebra morphism that commutes with the differentials. A cbba quasi-isomorphism is a cbba morphism $f:A\to B$ such that the induced maps $H_{BC}(f)$ and $H_A(f)$ are isomorphisms; according to \cite[Theorem C]{Stelzig_2025}, this is sufficient to ensure that the maps induced in all related cohomologies are isomorphisms. Two cbbas $A$ and $B$ are weakly equivalent if they are connected by a chain of quasi-isomorphisms
\[
	A\overset{f_1}{\longrightarrow}C_1\overset{f_2}{\longleftarrow}C_2\overset{f_3}{\longrightarrow}...\overset{f_{n-1}}{\longleftarrow}C_{n-1}\overset{f_n}{\longrightarrow}B.
\]

\begin{definition}\cite[Definition 2.9]{Stelzig_2025}
	Let $A$ be an augmented cbba.
	\begin{enumerate}
		\item $A$ is nilpotent if it is a composition of (bigraded) Hirsch extensions (see \cite[Definition 2.3]{Stelzig_2025});
		\item $A$ is minimal if it is nilpotent and $\text{Im}\ddbar\subseteq A^+\cdot A^+$;
		\item a cbba quasi-isomorphism $\varphi:\mathcal{M}\to A$, where $\mathcal{M}$ is a nilpotent cbba, is called a model for $A$. If $\mathcal{M}$ is minimal, $\varphi$ is called a minimal model.
	\end{enumerate}
\end{definition}

Equivalently, an augmented cbba $A$ is nilpotent if it admits a presentation as a free commutative bigraded algebra $A=\Lambda V$, where $V$ admits a well-ordered basis $\{v_i\}_{i\in I}$ such that $dv_j\in\Lambda\langle v_i|i<j\rangle$. A nilpotent augmented cbba $A$ is minimal if and only if the bicomplex of indecomposables $A^+/A^+\cdot A^+$ with induced differentials is minimal.

The existence of a minimal model is guaranteed for any manifold that is compact, connected and holomorphically simply connected (i.e., $H^1_A(X):=H^{1,0}_A(X)\oplus H^{0,1}_A(X)=0$).

\begin{theorem}\cite[Theorem 2.34]{Stelzig_2025}\label{thmexistence}
	For any connected, compact complex manifold that is holomorphically simply connected, there exists a simply connected, real, degree-wise finite dimensional bigraded minimal model for the cbba of forms, which is unique up to isomorphism.
\end{theorem}

\subsection{Strong formality} Following \cite[Definition 4.1]{milivojevic2024bigradednotionsformalityaepplibottchernmassey}, a cbba $A$ is said to be strongly formal if it is weakly equivalent to a cbba $(H,\partial_H,\bar{\partial}_H)$ with vanishing differentials, i.e. $\partial_H=\bar{\partial}_H=0$. Equivalently, a cbba is strongly formal if it is weakly equivalent to one whose underlying bicomplex decomposes solely into dots (see \cite{Stelzig_2021} for the decomposition of double complexes in dots, squares and zigzags). A complex manifold is strongly formal if its cbba of differential forms satisfies this property. 

Several classes of strongly formal manifolds have been identified: 
\begin{itemize}
	\item Hermitian symmetric spaces \cite[Proposition 4.19]{milivojevic2024bigradednotionsformalityaepplibottchernmassey};
	\item compact K\"ahler solvmanifolds \cite[Corollary 5.11]{sferruzza2024bottchernformalitymasseyproducts};
	\item splitting-type solvmanifolds of the form $\C^n\ltimes\C^m/\Gamma$ satisfying the $\ddbar$-Lemma \cite[Corollary 4.7]{RUBINI2025105586};
	\item complete smooth complex toric varieties and homogeneous compact K\"ahler manifolds \cite[Theorem 3.2.3 and Corollary 4.1.10]{placini2025strongformalitytorichomogeneous};
	\item compact geometrically Aeppli formal manifolds \cite[Proposition 1]{sferruzza2025hermitiangeometricallyformalmanifolds};
	\item compact K\"ahler manifolds of complex dimension $n\ge2$ with cohomology of complete intersection type \cite[Theorem 3.5]{Stelzig_2025}.
\end{itemize}

Conversely, in \cite{placini2024nontrivialmasseyproductscompact}, it is shown that any closed Riemann surface of genus $g\ge2$ is not strongly formal. Furthermore, the authors provide an example of a $3$-dimensional simply connected compact K\"ahler manifold that fails to be strongly formal. 

\section{$s$-strong formality of cbbas}\label{section2}

The goal of this section is to introduce the notion of $s$-strong formality and to establish several related results in a purely algebraic setting. 

We begin by observing that if a cbba $A$ is weakly equivalent to a cbba $B$ and both admit minimal models, then there exists a minimal cbba $\mathcal{M}$ and cbba quasi-isomorphisms $\psi:\mathcal{M}\to A$ and $\varphi:\mathcal{M}\to B$. This is achieved by iteratively applying the lifting property to the minimal models of the cbbas of the chain connecting $A$ and $B$ (following the approach in the proof of \cite[Corollary 2.29]{Stelzig_2025}) until we get the required quasi-isomorphisms. Consequently, if $A$ is a strongly formal cbba with minimal model $\mathcal{M}$, we may assume the existence of a diagram
\[
	\begin{tikzcd}
		&\arrow[ld](\mathcal{M},\partial_{\mathcal{M}},\bar{\partial}_{\mathcal{M}})\arrow[rd]&\\
		(A,\partial_A,\bar{\partial}_A)&&(H,0,0)
	\end{tikzcd}
\]
where the arrows are cbba quasi-isomorphisms. In particular, we can assume the existence of a cbba quasi-isomorphism $\psi:\mathcal{M}\to H_{BC}(\mathcal{M})\simeq H_{BC}(A)$ between the minimal model and its cohomology. By composing this morphism with the inverse of the induced map $H_{BC}(\psi)$, we may further assume that it induces the identity in Bott-Chern cohomology.

We now recall the characterization of the classical notion of formality proved in \cite{DeligneGriffiths1975} that inspired the notion of $s$-formality introduced in \cite{fernandezmunoz}.

\begin{theorem}\cite[Theorem 4.1]{DeligneGriffiths1975}\label{Deligne}
	A minimal cdga $\mathcal{M}$ is formal if and only if there exists a graded vector space of generators $V$, $\mathcal{M}=\Lambda V$, which decomposes as a direct sum $V=C\oplus N$ with $d(C)=0$, $d$ injective on $N$ and such that every closed element in the ideal $I(N)$ generated by $N$ is exact.
\end{theorem}

\begin{remark}\label{remark1}
	Note that if $V=C\oplus N$ with $d(C)=0$, the injectivity of $d$ on $N$ is equivalent to asking $C=\ker(d_{|_V})$. 
\end{remark}

We begin by proving an analogue of Theorem \ref{Deligne} for strongly formal minimal cbbas.

\begin{theorem}\label{characterization}
	A minimal cbba $(\mathcal{M}, \partial, \bar{\partial})$ is strongly formal if and only if it satisfies the $\ddbar$-Lemma and there exists a bigraded vector space of generators $V$, $\mathcal{M}=\Lambda V$, which decomposes as $V=C\oplus N$ as bigraded vector spaces such that:
	\begin{enumerate}[label=(\roman*)]
		\item for each $(p,q)$, $d(C^{p,q})=0$ (equivalently $\partial C^{p,q}=\bar{\partial}C^{p,q}=0$);
		\item for each $(p,q)$, $d$ is injective on $N^{p,q}$;
		\item \label{iiiCharac}every $\ddbar$-closed element in the ideal $I(N):=N\cdot(\Lambda V)$ belongs to $\partial\mathcal{M}+\bar{\partial}\mathcal{M}$;
		\item \label{ivCharac}every $\partial$ and $\bar{\partial}$-closed element in the ideal $I(N)$ belongs to $\ddbar\mathcal{M}$.
	\end{enumerate}
	\begin{remark}\label{remark}
		We point out that conditions \ref{iiiCharac} and \ref{ivCharac} are equivalent to asking that $I(N)$ does not contribute to the cohomology of $\mathcal{M}$, i.e., $H_*(i)=0$ for $*\in\{BC, A\}$, where $i:I(N)\hookrightarrow\mathcal{M}$ is the inclusion. Furthermore, under the assumption of the $\ddbar$-Lemma, condition \ref{iiiCharac} implies condition \ref{ivCharac}. Indeed, if $\eta\in I(N)$ is both $\partial$ and $\bar{\partial}$-closed, then it is also $\ddbar$-closed; by \ref{iiiCharac}, $\eta\in\text{Im}\partial+\text{Im}\bar{\partial}$ and, due to the $\ddbar$-Lemma, $\eta\in\ker\partial\cap\ker\bar{\partial}\cap(\text{Im}\partial+\text{Im}\bar{\partial})=\text{Im}\ddbar$, which is condition \ref{ivCharac}. We list both conditions for consistency with the upcoming definition of $s$-strong formality.
	\end{remark}
	\begin{proof}
		Suppose there exists a set of generators $V$ admitting a decomposition $V^{p,q}=C^{p,q}\oplus N^{p,q}$ as above. We want to construct a pluripotential quasi-isomorphism between the minimal cbba and its Bott-Chern cohomology. Define the projections $\pi^{p,q}:V^{p,q}\to C^{p,q}$ and a map $\psi:\Lambda V\to H_{BC}(\Lambda V)$ by setting $\psi(v)=[\pi^{p,q}(v)]_{BC}$ for $v\in V^{p,q}$, extending it to $\Lambda V$ by linearity and multiplicativity. By construction, $\psi$ is an algebra morphism. Note that $\Lambda V=\Lambda C\oplus I(N)$ and, for any $\alpha=c+n$ with $c\in \Lambda C$ and $n\in I(N)$, we have $\psi(\alpha)=[c]_{BC}$.
		
		We now verify that $\psi$ commutes with the differential, which is equivalent to showing that $\psi\circ d=0$. To this end, consider $d\alpha\in\Lambda V$ and write $d\alpha=\gamma+\eta$ with $\gamma\in\Lambda C$ and $\eta\in I(N)$. Since $\eta=d\alpha-\gamma$ is $\ddbar$-closed in $I(N)$, condition \ref{iiiCharac} implies that $\eta\in\text{Im}\partial+\text{Im}\bar{\partial}$. By the $\ddbar$-Lemma, it follows that
		\[
			\gamma=d\alpha-\eta\in(\text{Im}\partial+\text{Im}\bar{\partial})\cap\ker\partial\cap\ker\bar{\partial}=\text{Im}\ddbar.
		\]
		Thus, $\psi(d\alpha)=[\gamma]_{BC}=0$, confirming that $\psi$ is a cbba morphism.
		
		By construction, $\psi$ is a quasi-isomorphism. Indeed, it follows that $\psi$ induces the identity in both Bott-Chern and Aeppli cohomologies. To see this, consider an element $x=c+n$, with $c\in \Lambda C$ and $n\in I(N)$, which is respectively $\partial$ and $\bar{\partial}$-closed, or $\ddbar$-closed. In either case, $n$ is closed in the same sense; consequently, $n$ is either $\ddbar$-exact or, respectively, $n$ belongs to $\text{Im}\partial+\text{Im}\bar{\partial}$. It follows that $[x]=[c]$ and 
		\[
			H_{*}(\psi)[x]=[\psi(x)]=[c]=[x]
		\]
		for Bott-Chern cohomology or, respectively, for Aeppli cohomology.
		
		Conversely, assume that $(\Lambda V,\partial,\bar{\partial})$ is strongly formal. The $\ddbar$-Lemma is implicit by this condition. Our goal is to modify the spaces of generators $V^{p,q}$ into spaces $\hat{V}^{p,q}$ satisfying the required properties while ensuring that $\Lambda V^{\le i}=\Lambda \hat{V}^{\le i}$ for all $i\ge0$. Let $\psi:(\Lambda V,\partial,\bar{\partial})\to (H_{BC}(\Lambda V),\partial\equiv 0,\bar{\partial}\equiv0)$ be the cbba quasi-isomorphism realizing the strong formality. For each $p+q=i$, we define
		\[
			N^{p,q}=\ker\left(V^{p,q}\xrightarrow{\psi} H^{p,q}_{BC}(\Lambda V)\to\dfrac{H^{p,q}_{BC}(\Lambda V)}{\psi(\Lambda V^{<i})}\right).
		\]
		For any generator $x\in N^{p,q}$, we choose $a_x\in(\Lambda V^{<i})^{p,q}$ such that $\psi(x)=\psi(a_x)$. Such an element $a_x$ can be taken to be closed, as $\psi$ maps decomposables to decomposables in $H_{BC}(\Lambda V)$. We set $\hat{x}=x-a_x$, yielding a space $\hat{N}^{p,q}$ isomorphic to $N^{p,q}$ via the assignment $x\mapsto x-a_x$. Indeed, if $x-a_x=y-a_y$, since for $p+q=i$ we have $N^{p,q}\cap\Lambda V^{<i}=\{0\}$, it follows that $a_x=a_y$ and $x=y$. To construct $\hat{C}^{p,q}$, we define 
		\[
			C^{p,q}=\ker\left(V^{p,q}\xrightarrow{d}\Lambda V\to\dfrac{\Lambda V}{d(\Lambda V^{<i})}\right).
		\]
		For $y\in C^{p,q}$, there exists $b_y\in\Lambda V^{<i}$ such that $dy=db_y$. Since $b_y$ is defined up to closed elements, we may assume $\psi(b_y)=0$. We claim that $b_y$ can be chosen of pure bidegree $(p,q)$. Indeed, suppose
		\[
			b_y=\sum_{j=-q}^{p}b^{p-j,q+j},\quad \text{where }|b^{p-j,q+j}|=(p-j,q-j).
		\]
		Since $|y|=(p,q)$, degree considerations imply the following system:
		\begin{equation}
			\begin{cases}\label{cases}
				\partial y=\partial b^{p,q}+\bar{\partial}b^{p+1,q-1}\\
				\bar{\partial}y=\bar{\partial}b^{p,q}+\partial b^{p-1,q+1}\\
				\partial b^{p-j,q+j}+\bar{\partial} b^{p-j+1,q+j-1}=0\quad\forall j=-q,...,-1\\
				\bar{\partial}b^{p-j,q+j}+\partial b^{p-j-1,q+j+1}=0\quad\forall j=1,...,p\\
		\end{cases}
		\end{equation}
		Consider the pure bidegree element $y-b^{p,q}$. From the first two rows of \eqref{cases}, it follows that
		\[
			\partial(y-b^{p,q})=\bar{\partial}b^{p+1,q-1};\quad\bar{\partial}(y-b^{p,q})=\partial b^{p-1,q+1}.
		\]
		By the $\ddbar$-Lemma, $\partial(y-b^{p,q})$ and $\bar{\partial}(y-b^{p,q})$ are $\ddbar$-exact; thus, there exist $\phi^{p,q-1}$ and $\phi^{p-1,q}$ of pure bidegrees $(p,q-1)$ and $(p-1,q)$, respectively, such that $\partial(y-b^{p,q})=\ddbar\phi^{p,q-1}$ and $\bar{\partial}(y-b^{p,q})=\ddbar\phi^{p-1,q}$. Consequently, the element $y-b^{p,q}-\bar{\partial}\phi^{p,q-1}+\partial\phi^{p-1,q}$ is both $\partial$ and $\bar{\partial}$-closed. 
		
		To prove the claim, we show that $\bar{\partial}\phi^{p,q-1}-\partial\phi^{p-1,q}$ is decomposable. Suppose $\bar{\partial}\phi^{p,q-1}-\partial\phi^{p-1,q}=\alpha+\beta$ with $\alpha\in V^{p,q}$ and $\beta\in\Lambda V^{<i}$. Since $\psi$ is a cbba morphism, $\psi(\alpha+\beta)=\psi(\bar{\partial}\phi^{p,q-1}-\partial\phi^{p-1,q})=0$. Then $\psi(\alpha)=\psi(-\beta)$, which implies $\alpha\in N^{p,q}$. It follows that $y-\alpha\in C^{p,q}$, yielding $\alpha\in C^{p,q}\cap N^{p,q}$. As established below, this intersection is trivial; hence $\alpha=0$, and $\bar{\partial}\phi^{p,q-1}-\partial\phi^{p-1,q}$ is indeed decomposable. 
		
		Let $\tilde{b}:=b^{p,q}+\bar{\partial}\phi^{p,q-1}-\partial\phi^{p-1,q}$ be the decomposable element of pure bidegree $(p,q)$. We have $\bar{\partial}\tilde{b}=\bar{\partial}b^{p,q}+\ddbar\phi^{p-1,q}=\bar{\partial}y$ and $\partial\tilde{b}=\partial b^{p,q}+\ddbar\phi^{p,q-1}=\partial y$. Choosing $b_y$ to be $\tilde{b}$ and setting $\hat{y}=y-b_y$, we obtain a space $\hat{C}^{p,q}$ isomorphic to $C^{p,q}$. By construction, $\psi(\hat{N}^{p,q})=0$ and $d(\hat{C}^{p,q})=0$.
		
		We observe that $\hat{N}^{p,q}\cap\hat{C}^{p,q}=\{0\}$: if $x\in\hat{N}^{p,q}\cap\hat{C}^{p,q}$, then 
		\[
			\psi^*[x]_{BC}=\psi(x)=0\quad\iff\quad[x]_{BC}=0.
		\]
		Thus, $x$ would be a generator in $\text{Im}\ddbar$, which is a contradiction unless $x=0$. 
		
		Furthermore, we have $V^{p,q}=C^{p,q}\oplus N^{p,q}$: if $x\in C^{p,q}\cap N^{p,q}$, then $\psi(x-a_x)=0$ and $d(x-b_x)=0$. As previously noted, we can choose $a_x$ and $b_x$ such that $d(a_x)=0$ and $\psi(b_x)=0$. Thus, $\psi(x-a_x-b_x)=0$ and $d(x-a_x-b_x)=0$, yielding $[x-a_x-b_x]_{BC}=0$. Since $x$ is a generator, $x-a_x-b_x$ cannot be $\ddbar$-exact, forcing $x=0$. Finally, to show that $V^{p,q}=C^{p,q}+N^{p,q}$, let $x\in V^{p,q}$. If $\psi(x)=[t]_{BC}$, we have $\psi(x-t)=0$. Write $t=t_1+t_2$, with $t_1\in V^{p,q}$ and $t_2\in(\Lambda V^{<p+q})^{p,q}$. Then $t_1\in C^{p,q}$ because $dt_1=d(-t_2)\in d(\Lambda V^{<p+q})$. Since $\psi(x-t_1)=\psi(t_2)$, we have $x-t_1\in N^{p,q}$, concluding that $x=x-t_1+t_1\in C^{p,q}+N^{p,q}$.
		
		The injectivity of $d$ on $\hat{N}^{p,q}$ follows from Remark \ref{remark1}. To establish the other properties, consider a closed element $\alpha\in I(\bigoplus_{p,q}\hat{N}^{p,q})$ of pure bidegree. Such an element defines a Bott-Chern cohomology class $[\alpha]_{BC}$. Since $\psi$ vanishes on $\hat{N}$ by construction, we have $\psi^*([\alpha]_{BC})=[\psi(\alpha)]_{BC}=0$. Given that $\psi$ induces the identity in Bott-Chern cohomology, it follows that $[\alpha]_{BC}=0$, which implies that $\alpha$ is $\ddbar$-exact. 
		
		Lastly, Let $\alpha\in I(\bigoplus_{p,q}\hat{N}^{p,q})$ be a $\ddbar$-closed element of pure bidegree; this element defines an Aeppli cohomology class $[\alpha]_A$. In an analogous way, $H_A(\psi)[\alpha]_A=[\psi(\alpha)]_A=0$. Since $\psi$ induces an isomorphism in Aeppli cohomology, we conclude that $[\alpha]_A=0$, thereby ensuring that $\alpha\in\text{Im}\partial+\text{Im}\bar{\partial}$.
	\end{proof}
\end{theorem}

We adopt the following notation: if $A$ is a subalgebra of a cbba $B$, we denote by $\mathcal{C}(A)$ the smallest sub-cbba of $B$ containing $A$. Equivalently, $\mathcal{C}(A)$ is the subalgebra of $B$ generated by $A$, $\partial A$, $\bar{\partial}A$ and $\ddbar A$.

To verify the validity of the $\ddbar$-Lemma relatively to the degree of the generators, we introduce the following notion.

\begin{definition}\label{ddbaruptos}
	 Let $(\mathcal{M}=\Lambda V, \partial,\bar{\partial})$ be a minimal cbba and let $s\in\mathbb{N}$. $\mathcal{M}$ is said to satisfy the $\ddbar$-Lemma {\em"up to degree $s$"} if every $d$-exact, $\partial$ and $\bar{\partial}$-closed element in $\mathcal{C}(\Lambda(V^{\le s}))$ is $\ddbar$-exact in $\Lambda V$
\end{definition}

In light of Theorem \ref{characterization}, we define $s$-strong formality, which serves as the complex analogue of $s$-formality introduced in \cite[Definition 2.2]{fernandezmunoz}. 

\begin{definition}\label{sstrong}
	Let $(\mathcal{M}, \partial,\bar{\partial})$ be a minimal cbba and let $s\in\mathbb{N}$. $\mathcal{M}$ is {\em$s$-strongly formal} if there exists a set of generators $V$, $\mathcal{M}=\Lambda V$, such that, for each $p+q\le s$, we have a splitting $V^{p,q}=C^{p,q}\oplus N^{p,q}$ satisfying:
	\begin{enumerate}[label=(\roman*)]
		\item $d(C^{p,q})=0;$
		\item $d$ is injective on $N^{p,q}$;
		\item \label{iii} every $\partial$ and $\bar{\partial}$-closed element in the ideal $I_s:=I(N^{\le s}+\partial N^{\le s}+\bar{\partial}N^{\le s}+\ddbar N^{\le s})\subseteq\mathcal{C}(\Lambda(V^{\le s}))$ is $\ddbar$-exact in $\Lambda V$;
		\item \label{(iv)} every $\ddbar$-closed element in the ideal $I_s$ belongs to $\text{Im}\partial+\text{Im}\bar{\partial}\subseteq\Lambda V$;
		\item $\mathcal{M}$ satisfies the $\ddbar$-Lemma "up to degree $s$".
	\end{enumerate} 
A complex manifold $M$ is $s$-strongly formal if its cbba of differential forms admits a bigraded minimal model that is $s$-strongly formal.
\end{definition}

Note that $\mathcal{C}(\Lambda V^{\le s})\subseteq \Lambda V^{\le s+1}$ due to the minimality of $\Lambda V$. As observed in Remark \ref{remark}, conditions \ref{iii} and \ref{(iv)} imply that elements in the ideal $I_s$ do not contribute to the cohomology of the entire cbba. 

As a first result on $s$-strong formality, we prove the following analogue of \cite[Lemma 2.7]{fernandezmunoz}, proceeding similarly to the proof of Theorem \ref{characterization}.

\begin{lemma}\label{lemma}
	Let $(\Lambda V,\partial,\bar{\partial})$ be a minimal cbba. Then $\Lambda V$ is $s$-strongly formal if and only if it satisfies the $\ddbar$-Lemma "up to degree $s$" and there exists a cbba morphism
	\begin{equation}\label{equation2}
		\psi:(\mathcal{C}(\Lambda V^{\le s}),\partial,\bar{\partial})\to(H_{BC}(\Lambda V),\partial=0,\bar{\partial}=0)
	\end{equation}
	such that the maps induced by $\psi$ in Bott-Chern and Aeppli cohomologies coincide with the maps induced by the inclusion 
	\[
		i:(\mathcal{C}(\Lambda V^{\le s}),\partial,\bar{\partial})\hookrightarrow(\Lambda V,\partial,\bar{\partial}).
	\]
	\begin{proof}
		Assume $\Lambda V$ is $s$-strongly formal. For $p+q\le s$, we define the map $\psi$ on the generators $V^{p,q}=C^{p,q}\oplus N^{p,q}$ as $\psi(v)=[\pi^{p,q}(v)]_{BC}$, where $\pi^{p,q}:V^{p,q}\to C^{p,q}$ is the projection. If $\eta\in\partial N^s$ or $\eta\in\bar{\partial}N^s$ is indecomposable, we set $\psi(\eta)=0$. By extending $\psi$ to the whole cbba $\mathcal{C}(\Lambda V^{\le s})$ by multiplicativity and linearity, we thus obtain a well-defined morphism of algebras. This is sufficient to have $\psi(\partial N^s)=\psi(\bar{\partial}N^s)=0$ even in the decomposable case: suppose $\partial\eta\in\partial N^s$ is of pure bidegree and decomposable. We can write $\partial\eta=\gamma+\beta$, where $\gamma\in\Lambda C^{\le s}$, $\beta\in I_s$ are decomposable. This implies that $\gamma=\partial\eta-\beta$ is $\partial$ and $\bar{\partial}$-closed and lies in $I_s$, hence it is $\ddbar$-exact by property \ref{iii}. Thus, by linearity and multiplicativity, $\psi(\partial\eta)=[\gamma]_{BC}=0$. The proof then proceeds as in Theorem \ref{characterization} to ensure commutativity with the differentials and to verify that the induced maps in Bott-Chern and Aeppli cohomologies coincide with the ones induced by the inclusion.
		
		Conversely, suppose there exists a morphism $\psi$ as in \eqref{equation2}. Following the arguments in the proof of Theorem \ref{characterization}, one can use $\psi$ instead of the quasi-isomorphism realizing the strong formality to construct the spaces $C^{p,q}$ and $N^{p,q}$ for $p+q\le s$, and, consequently, the spaces $\hat{C}^{p,q}$ and $\hat{N}^{p,q}$ verifying $V^{p,q}=C^{p,q}\oplus N^{p,q}$ and $\mathcal{C}(\Lambda V^{\le s})=\mathcal{C}(\Lambda \hat{V}^{\le s})$.
		
		The injectivity of $d$ on $\hat{N}^{p,q}$ follows from Remark \ref{remark1}. To establish property \ref{iii}, consider a $\partial$ and $\bar{\partial}$-closed element $\alpha\in I_s$ of pure bidegree. Since $\psi$ vanishes on $\hat{N}$ and commutes with the differentials, we have $H_{BC}(\psi)[\alpha]_{BC}=[\psi(\alpha)]_{BC}=0$. Given that $\psi$ induces the same maps as the inclusion $i$ in cohomology, this means exactly $\alpha\in\ddbar(\Lambda V)$. In an analogous way, condition \ref{(iv)} is verified.
	\end{proof}
\end{lemma}

This Lemma provides an alternative definition of $s$-strong formality that is more closely aligned with the definition of strong formality in \cite[Definition 4.1]{milivojevic2024bigradednotionsformalityaepplibottchernmassey}. Note that the morphism $\psi$ defined in Lemma \ref{lemma} induces isomorphisms in Bott-Chern and Aeppli cohomologies in degrees $\le s$ and injections in degree $s+1$.

We aim to prove a result analogous to \cite[Theorem 3.1]{fernandezmunoz}, which establishes the Sullivan formality of orientable, compact $(n-1)$-formal manifolds of dimension $2n$ or $2n-1$. To this end, we introduce the notion of Serre-dual cbba.

\begin{definition}\label{SD}
	An augmented cbba $(B,\partial,\bar{\partial})$ is said to be {\em$n$-Serre-dual} ($n$-SD) if it is connected and its cohomologies satisfy the following properties:
	\begin{enumerate}[label=(\roman*)]
		\item \label{iSD} for $*\in\{BC, A\}$, $H_*(B)$ is degree-wise finite dimensional;
		\item \label{iiSD} for $*\in\{BC, A\}$, $H_*^{p,q}(B)=0$ if $p<0$, $q<0$, $p>n$ or $q>n$; $H_{A}^{n,n}(B)\simeq\C$;
		\item \label{iiiSD} for each $0\le p,q\le n$, the multiplication induces a perfect pairing between Bott-Chern and Aeppli cohomologies:
		\[ 
			H_{BC}^{p,q}(B)\times H^{n-p,n-q}_A(B)\to H^{n,n}_A(B).
		\]
	\end{enumerate}
\end{definition}

The following Proposition provides a first relation between the degree-dependent notions introduced in this section and the global properties of an $n$-SD cbba.

\begin{proposition}\label{ddbarlemma}
	Let $\Lambda V$ be a minimal $n$-SD cbba. If $\Lambda V$ satisfies the $\ddbar$-Lemma "up to degree $(n-1)$", then it satisfies the global $\ddbar$-Lemma.
	\begin{proof}
		We show that condition $(a)_k$ of \cite[Lemma 5.15]{DeligneGriffiths1975} is verified for each $k\le n$, namely that $\ker\partial\cap\ker\bar{\partial}\cap\text{Im}d=\text{Im}\ddbar$ holds on $(\Lambda V)^k$. 
		
		Let $d\eta\in(\Lambda V)^k$ be a $d$-exact, $\partial$ and $\bar{\partial}$-closed element. For $k\le n-1$, it is $\ddbar$-exact by the $\ddbar$-Lemma "up to degree $n-1$". If $k=n$, then $\eta\in (\Lambda V)^{n-1}=(\Lambda V^{\le n-1})^{n-1}$, which implies $d\eta\in\mathcal{C}(\Lambda V^{\le n-1})$. Thus, we can still apply the $\ddbar$-Lemma "up to degree $n-1$" to conclude that $d\eta\in\text{Im}\ddbar$. This proves condition $(a)_k$ for $k\le n$. 
		
		By the equivalent conditions in \cite[Lemma 5.15]{DeligneGriffiths1975}, the natural map $H_{BC}^{p,q}(\Lambda V)\to H_{A}^{p,q}(\Lambda V)$ induced by the identity is an isomorphism for $p+q\le n-1$. For $p+q=n$, condition \ref{iiiSD} of Definition \ref{SD} yields the isomorphism
		\[
			H_{BC}^{p,q}(\Lambda V)\simeq H_{A}^{n-p,n-q}(\Lambda V)=H_{A}^{q,p}(\Lambda V).
		\]
		Thus, we obtain $\sum_{p+q=k}h_{BC}^k(\Lambda V)=\sum_{p+q=k}h_{A}^k(\Lambda V)$ for all $k\le n$. By the duality defined by the non-degeneracy of the pairing, this identity is verified for all degrees $k$ with $k\le2n$. By properties \ref{iSD} and \ref{iiSD}, we may apply \cite[Theorem 3.1]{AngTar16} to the $n$-SD minimal cbba $\Lambda V$, which implies the $\ddbar$-Lemma "up to degree $2n$". 
		
		Finally, this condition coincides with the global $\ddbar$-Lemma on an $n$-SD minimal cbba. Indeed, let $d\eta\in\ker\partial\cap\ker\bar{\partial}$. If $|d\eta|\le2n$, the $\ddbar$-Lemma "up to degree $2n$" implies $d\eta\in\text{Im}\ddbar$. If $|d\eta|>2n$, then $d\eta$ defines a cohomology class in $H^{|d\eta|}_{BC}(\Lambda V)$. Since $\Lambda V$ is an $n$-SD cbba, its Bott-Chern cohomology vanishes in degrees above $2n$; hence, we conclude $d\eta\in\text{Im}\ddbar$.
	\end{proof}
\end{proposition}

We now provide an analogue of \cite[Lemma 2.10]{fernandezmunoz} within the framework of bigraded minimal models and strong formality.

\begin{proposition}\label{prop}
	Let $\Lambda V$ be a minimal $n$-SD cbba. Then $\Lambda V$ is strongly formal if and only if it is $2n$-strongly formal.
	\begin{proof}
		The right implication follows from the characterization provided in Theorem \ref{characterization}. Indeed, strong formality implies $s$-strong formality for any $s\le 2n$: the global $\ddbar$-Lemma trivially implies the $\ddbar$-Lemma "up to degree $s$". To verify property \ref{iii} of Definition \ref{sstrong} for $s\le2n$, let $\eta\in I_s\cap\ker\partial\cap\ker\bar{\partial}$, 
		\[
			\eta=\eta_0\cdot v_0+\partial\eta_1\cdot v_1+\bar{\partial}\eta_2\cdot v_2+\ddbar\eta_3\cdot v_3,\quad\eta_i\in N^{\le s}, v_i\in\mathcal{C}(\Lambda( V^{\le s})).
		\] 
		As in the proof of Theorem \ref{characterization}, we may choose a cbba morphism $\psi$ realizing the strong formality such that $\psi(N)=0$ and which induces the identity in Bott-Chern cohomology. It follows that $\psi(\eta)=[\eta]_{BC}$ and, since $\psi$ is a cbba morphism, we have  
		\[
			\psi(\eta)=\psi(\eta_0)\cdot \psi(v_0)+\partial\psi(\eta_1)\cdot \psi(v_1)+\bar{\partial}\psi(\eta_2)\cdot\psi( v_2)+\ddbar\psi(\eta_3)\cdot \psi(v_3)=0,
		\]
		hence $\eta\in\text{Im}\ddbar$. To verify condition \ref{ivCharac}, let $\eta\in I_s\cap\ker\ddbar$. By the $\ddbar$-Lemma, we can write $\eta=\partial\alpha+\bar{\partial}\beta+\xi$, with $\xi\in\ker\partial\cap\ker\bar{\partial}$. Thus, since $\psi(I_s)=0$ as shown above and $\psi$ commutes with differentials, we obtain
		\[
			\psi(\xi)=\psi(\eta-\partial\alpha-\bar{\partial}\beta)=0.
		\]
		Therefore, since $\psi$ on a $\partial$ and $\bar{\partial}$-closed element coincides with its Bott-Chern cohomology class and  $[\xi]_{BC}=[\eta-\partial\alpha-\bar{\partial}\beta]_{BC}$, we conclude $\xi\in\text{Im}\ddbar$. This yields $\eta\in\text{Im}\partial+\text{Im}\bar{\partial}$, thus proving condition \ref{ivCharac}.
		
		Conversely, suppose that $\Lambda V$ is $2n$-strongly formal. We set $N^{p,q}=V^{p,q}$ and $C^{p,q}=0$ for all $p+q>2n$ and we denote by $N=\bigoplus_{p+q\ge0} N^{p,q}$. To prove strong formality, it suffices to verify the following conditions:
		\begin{enumerate}
			\item\label{(1)} if $x\in I(N)$ is closed, then $x$ is $\ddbar$-exact;
			\item\label{(2)} if $x\in I(N)$ is $\ddbar$-closed, then $x\in\text{Im}\partial+\text{Im}\bar{\partial}$;
			\item\label{(3)} the $\ddbar$-Lemma holds.
		\end{enumerate} 
		
		To prove (\ref{(1)}), let $x\in I(N)\cap\ker\partial\cap\ker\bar{\partial}$. If $\deg(x)\le2n$, then $x\in I_{2n}$ by degree reasons, thus the claim follows from property \ref{(iv)} of Definition \ref{sstrong} for $s=2n$. If $\deg(x)=i>2n$, then $x$ defines a cohomology class in $H^i_{BC}(\Lambda V)$. Since $\Lambda V$ is an $n$-SD cbba, its Bott-Chern cohomology vanishes in degrees $i>2n$, hence the element $x$ is $\ddbar$-exact. Condition (\ref{(2)}) is verified analogously.
		
		Finally, the $\ddbar$-Lemma "up to degree $2n$" implies the validity of the $\ddbar$-Lemma "up to degree $(n-1)$", hence due to Proposition \ref{ddbarlemma} condition (\ref{(3)}) follows.
	\end{proof}
\end{proposition}

We now establish an analogue to \cite[Theorem 3.1]{fernandezmunoz} in the bigraded setting. In this direction, we first state some preliminary technical Lemmas. We observe that if $\Lambda V$ is a minimal, connected and simply connected cbba, the minimality condition allows us to order the generators of $V^k$ as $x_1,...,x_{n_k}$ such that $dx_i\in\Lambda (V^{\le k-1}\oplus\langle x_1,..,x_{i-1}\rangle)\oplus V^{k+1}$. Furthermore, considering the composition 
\begin{equation}\label{rho}
	\rho: V^k\overset{d}{\longrightarrow}\mathcal{C}(\Lambda V^{\le k})\overset{\pi}{\longrightarrow}\frac{\mathcal{C}(\Lambda V^{\le k})}{d(\mathcal{C}(\Lambda V^{\le k-1}))},
\end{equation}
we may choose $x_1,...,x_p$ to be a basis for $\ker\rho$, as the minimality condition is trivially satisfied on this subspace.

\begin{lemma}\label{technical1}
	Let $\mathcal{M}=\Lambda V$ be a simply connected minimal $n$-SD cbba satisfying the $\ddbar$-Lemma. Let the generators $V^k=\langle x_1,...,x_{n_k}\rangle$ of pure bidegree be ordered as above, with $x_1,...,x_p$ spanning $\ker\rho$. For $i> p$, let $V_{i-1}:=V^{\le k-1}\oplus\langle x_1,...,x_{i-1}\rangle$ and assume that $V_{i-1}$ admits a splitting $C_{i-1}\oplus N_{i-1}$ satisfying the properties of $k$-strong formality with respect to the ideal 
	\[
		I_{i-1}(N):=(N_{i-1}+\partial N_{i-1}+\bar{\partial}N_{i-1}+\ddbar N_{i-1})\cdot\mathcal{C}(\Lambda V_{i-1}).
	\]
	If $k\ge n$, or $k=n-1$ and $\partial x_i$, $\bar{\partial}x_i\notin\ker\left(\mathcal{M}^+\to \mathcal{M}^+/\mathcal{M}^+\cdot \mathcal{M}^+\right)$, then any $\ddbar$-closed element $\eta$ of pure bidegree $(n,n)$ in the ideal 
	\begin{equation}\label{ideal_i}
		I_i(N)=I_{i-1}(N)+\langle x_i,\partial x_i,\bar{\partial}x_i,\partial\bar{\partial}x_i\rangle\cdot\mathcal{C}(\Lambda (V_{i-1} \oplus \langle x_i \rangle))
	\end{equation}
	can be written as $\eta=\eta_0+\tau\cdot x_i+\beta$, where $\eta_0,\tau\in I_{i-1}(N)$, $\beta\in\text{Im}\partial+\text{Im}\bar{\partial}$ and $\tau$ is both $\partial$ and $\bar{\partial}$-closed.
	\begin{proof}
		The proof is divided into two steps: we first prove the result for $k\ge n$, and subsequently for $k=n-1$ under the assumption that the differentials $\partial x_i$, $\bar{\partial}x_i$ are non-decomposable, i.e. $\partial x_i$, $\bar{\partial}x_i\notin\ker\left(\mathcal{M}^+\to \mathcal{M}^+/\mathcal{M}^+\cdot\mathcal{M}^+\right)$.\\
		
		\noindent\textbf{Step 1: $k\ge n$.}	By degree considerations, the only case in which there is a non-trivial power of $x_i$ is $\eta=\eta_0+\eta_1\cdot x_i+\eta_2\cdot\partial x_i+\eta_3\cdot\bar{\partial}x_i+\eta_4\cdot\ddbar x_i+\lambda x_i^2$, where $\lambda\in\C$ and $|x_i|=(n/2,n/2)$ (hence with $n$ even). Comparing the elements multiplying $x_i$ in the equation $\ddbar\eta=0$, we obtain $\ddbar\eta_1=2\lambda\ddbar x_i$. 
		
		If $\lambda\neq0$, the element $x_i-\frac{1}{2\lambda}\eta_1$ is $\ddbar$-closed; hence, by the $\ddbar$-Lemma, it can be written as $x_i-\frac{1}{2\lambda}\eta_1=c+\partial\alpha+\bar{\partial}\beta$ for some $\alpha,\beta,c\in \Lambda V$ with $dc=0$. This implies $\partial(x_i-\frac{1}{2\lambda}\eta_1)=\ddbar\beta$ and $\bar{\partial}(x_i-\frac{1}{2\lambda}\eta_1)=\ddbar(-\alpha)$, which would mean $dx_i\in d(\mathcal{C}(\Lambda V^{\le n-1}))$, contradicting the hypothesis $x_i\notin\ker\rho$. Consequently, we must have $\lambda=0$ and $\eta_1$ is $\ddbar$-closed.
		
		Now, consider a $\ddbar$-closed element $\eta$ of pure bidegree $(n,n)$:
		\begin{equation}\label{etaClaim1}
			\eta=\eta_0+\eta_1\cdot x_i+\eta_2\cdot\partial x_i+\eta_3\cdot\bar{\partial}x_i+\eta_4\cdot\ddbar x_i.
		\end{equation}
		By the minimality and the simply connectedness of the model, $\ddbar x_i\in \Lambda V_{i-1}$, which implies $\eta_4\cdot\ddbar x_i\in I_{i-1}(N)$. We may thus absorb this term into $\eta_0$ and assume, without loss of generality, that $\eta_4=0$. We then proceed by distinguishing three cases, based on the behavior of the differentials $\partial x_i$ and $\bar{\partial}x_i$.\\
		
		\noindent\textbf{Case 1.1: neither $\partial x_i$ nor $\bar{\partial}x_i$ are decomposable.} By comparing the terms multiplying $\partial x_i$ and $\bar{\partial}x_i$ in the relation $\ddbar\eta=0$, we obtain 
		\[
			\bar{\partial}\eta_1=(-1)^k\ddbar\eta_2,\quad\partial\eta_1=(-1)^{k-1}\ddbar\eta_3. 
		\]
		Adding and subtracting $(-1)^{k-1}\partial\eta_2\cdot x_i$ and $(-1)^{k-1}\bar{\partial}\eta_3\cdot x_i$ within \eqref{etaClaim1}, $\eta$ may be rewritten as
		\[
			\eta=\eta_0+\tau\cdot x_i+(-1)^{k-1}\partial(\eta_2\cdot x_i)+(-1)^{k-1}\bar{\partial}(\eta_3\cdot x_i),
		\]
		where $\tau:=\eta_1-(-1)^{k-1}\partial\eta_2-(-1)^{k-1}\bar{\partial}\eta_3$ is both $\partial$ and $\bar{\partial}$-closed as desired.\\
		
		\noindent\textbf{Case 1.2: only one of $\{\partial x_i,\bar{\partial}x_i\}$ is non-decomposable.} Assume, without loss of generality, that $\partial x_i\notin\ker\left(\mathcal{M}^+\to \mathcal{M}^+/\mathcal{M}^+\cdot\mathcal{M}^+\right)$ and $\bar{\partial}x_i$ is decomposable. We absorb $\eta_3\cdot\bar{\partial}x_i$ into $\eta_0$ and consider $\ddbar$-closed elements of the form $\eta=\eta_0+\eta_1\cdot x_i+\eta_2\cdot\partial x_i$.	As in the previous case, we have $\bar{\partial}\eta_1=(-1)^k\ddbar\eta_2$. Thus, adding and subtracting $(-1)^{k-1}\partial\eta_2\cdot x_i$, we may write
		\[
			\eta=\eta_0+\tau\cdot x_i+(-1)^{k-1}\partial(\eta_2\cdot x_i)
		\]
		with $\tau:=\eta_1-(-1)^{k-1}\partial\eta_2$. The element $\tau$ is therefore $\bar{\partial}$-closed. By \cite[Lemma 5.15]{DeligneGriffiths1975}, the $\ddbar$-Lemma implies the existence of a $\bar{\partial}$-exact element $\bar{\partial}\beta$ such that $\tau+\bar{\partial}\beta$ is $\partial$-closed (and still $\bar{\partial}$-closed). Since $k\ge n$, $\beta$ satisfies $|\beta|\le n-1$, ensuring that $\partial\beta\in\mathcal{C}(\Lambda V_{i-1})$. Adding and subtracting $\bar{\partial}\beta\cdot x_i$, we obtain
		\[
			\eta=\eta_0+(\tau+\bar{\partial}\beta)\cdot x_i-\bar{\partial}\beta\cdot x_i+(-1)^{k-1}\partial(\eta_2\cdot x_i).
		\]
		The thesis follows by observing $\bar{\partial}\beta\cdot x_i=\bar{\partial}(\beta\cdot x_i)-(-1)^{k-1}\beta\cdot\bar{\partial}x_i$.\\
		
		\noindent\textbf{Case 1.3: both $\partial x_i$ and $\bar{\partial}x_i$ are decomposable.} In this case, we absorb both $\eta_2\cdot\partial x_i$ and $\eta_3\cdot\bar{\partial}x_i$ into $\eta_0$ and consider $\eta=\eta_0+\eta_1\cdot x_i$. Since $\eta_1$ is $\ddbar$-closed, \cite[Lemma 5.15]{DeligneGriffiths1975} guarantees the existence of elements $\partial\alpha$ and $\bar{\partial}\beta$ such that $\eta_1+\partial\alpha+\bar{\partial}\beta$ is both $\partial$ and $\bar{\partial}$-closed. Proceeding as above, we conclude the proof of Step 1.\\
		
		\noindent\textbf{Step 2: $k=n-1$ and both $\partial x_i$ and $\bar{\partial}x_i$ are non-decomposable.}	Since the cbba $M$ has no $1$-degree generators, we may write $\eta$ as
		\begin{align*}
			\eta=&\eta_0+x_i\cdot(\eta_1+x_i\cdot\alpha+t\ddbar x_i)+\partial x_i\cdot(\eta_2+a\partial x_i+b\bar{\partial}x_i)\\
			&+\bar{\partial}x_i\cdot(\eta_3+c\bar{\partial}x_i)+\ddbar x_i\cdot\eta_4,
		\end{align*}
		where $\eta_0\in I_{i-1}$, $\alpha,\eta_j\in\mathcal{C}(\Lambda V_{i-1})$ for $j=1,2,3,4$ and $a,b,c,t\in\C$. We observe that $\partial x_i\cdot\partial x_i=\partial(x_i\cdot\partial x_i)$ and $\bar{\partial}x_i\cdot\bar{\partial}x_i=\bar{\partial}(x_i\cdot\bar{\partial}x_i)$; consequently, these terms lie in $\text{Im}\partial+\text{Im}\bar{\partial}$. We now distinguish two cases based on the parity of $n$.\\
		
		\noindent\textbf{Case 2.1: $n$ is even.} In this case, $x_i\cdot x_i=0$, hence we may assume $\alpha=0$. Furthermore, the following identities hold:
		\begin{align*}
			\partial(x_i\cdot\bar{\partial}x_i)&=\partial x_i\cdot\bar{\partial}x_i-x_i\cdot\ddbar x_i;\\
			\bar{\partial}(x_i\cdot\partial x_i)&=\partial x_i\cdot\bar{\partial}x_i+x_i\cdot\ddbar x_i.
		\end{align*}
		It follows that both $\partial x_i\cdot\bar{\partial}x_i$ and $x_i\cdot\ddbar x_i$ belong to $\text{Im}\partial+\text{Im}\bar{\partial}$. Thus, we may rewrite $\eta$ as
		\begin{equation}\label{eta}
			\eta=\eta_0+x_i\cdot\eta_1+\partial x_i\cdot\eta_2+\bar{\partial}x_i\cdot\eta_3+\ddbar x_i\cdot\eta_4+\partial\sigma+\bar{\partial}\xi
		\end{equation}
		which allows us to proceed as in Case 1.1.\\
		
		\noindent\textbf{Case 2.2: $n$ is odd.}	Since $\ddbar(x_i\cdot x_i)=2\partial x_i\cdot\bar{\partial}x_i+2x_i\cdot\ddbar x_i$, we may assume $b=0$ up to a $\ddbar$-exact element. Moreover, since $\ddbar x_i\in\Lambda V_{i-1}$, we can absorb $t\ddbar x_i$ into $\eta_1$. In the expression for $\ddbar\eta$, the only term with $\partial x_i\cdot\bar{\partial}x_i$ as a factor is $2\partial x_i\cdot\bar{\partial}x_i\cdot\alpha$. Since both $\partial x_i$ and $\bar{\partial}x_i$ are non-decomposable, the relation $\ddbar\eta=0$ implies $\alpha=0$. Thus, we can write $\eta$ as in \eqref{eta}, and the proof concludes as in Case 1.1.	
	\end{proof}
\end{lemma}
	
\begin{lemma}\label{technical2}
	Under the same hypotheses as in Lemma \ref{technical1}, it is possible to modify the generator $x_i$ by setting $\hat{x}_i:=x_i-\psi_i$ for some $\psi_i\in\mathcal{C}(\Lambda V_{i-1})$ such that any $\ddbar$-closed pure bidegree $(n,n)$ element in the ideal defined by \eqref{ideal_i} with $\hat{x}_i$ in place of $x_i$ belongs to $\text{Im}\partial+\text{Im}\bar{\partial}$. 
	\begin{proof}
		Let $k=|x_i|$. By Lemma \ref{technical1}, we can write a $\ddbar$-closed element $\eta$ in the ideal defined by \eqref{ideal_i} as $\eta=\eta_0+\tau\cdot x_i+\beta$, with $\eta_0\in I_{i-1}(N)$, $\tau\in\mathcal{C}(\Lambda V_{i-1})\cap\ker\partial\cap\ker\bar{\partial}$, $\beta\in\text{Im}\partial+\text{Im}\bar{\partial}$.
		
		First, suppose $[\tau]_{BC}=0$ (and then $[\tau]_A=0$). We write $\tau=\ddbar\gamma$ with $\gamma\in\mathcal{C}(\Lambda V_{i-1})$ for degree considerations. We add and subtract the necessary terms to have $\ddbar(\gamma\cdot x_i)$ in the expression of $\eta$. A direct computation yields
		\begin{align*}
			\eta=&\eta_0+\gamma\cdot\ddbar x_i+\ddbar(\gamma\cdot x_i)-(-1)^{k+1}\bar{\partial}(\gamma\cdot\partial x_i)\\
			&-(-1)^{k}\partial(\gamma\cdot\bar{\partial} x_i)+(-1)^{k+1}(\partial(\eta_2\cdot x_i)+\bar{\partial}(\eta_3\cdot x_i)).
		\end{align*}
		Since $\eta_0+\gamma\cdot\ddbar x_i\in I_{i-1}(N)$ is $\ddbar$-closed, the hypothesis on $I_{i-1}(N)$ ensures it lies in $\text{Im}\partial+\text{Im}\bar{\partial}$. Thus $\eta\in\text{Im}\partial+\text{Im}\bar{\partial}$. 
		
		Next, assume $[\tau]_{BC}\neq0$. Following the argument in the proof of \cite[Theorem 3.1]{fernandezmunoz}, we aim to modify $x_i$ with an element $\psi_i\in\mathcal{C}(\Lambda V_{i-1})^k$ such that, for $\hat{x}_i=x_i-\psi_i$, each $\ddbar$-closed element $\eta=\eta_0+\tau\cdot\hat{x}_i$ lies in $\text{Im}\partial+\text{Im}\bar{\partial}$. 
		
		Let $|x_i|=(p,q)$. Consider the collection of $\ddbar$-closed elements $z_j\in(\Lambda V)^{n-p,n-q}$ for which there exists $k_j\in (I_{i-1}(N))^{n,n}$ such that $k_j+z_j\cdot x_i$ is $\ddbar$-closed (this condition is verified up to elements in $\text{Im}\partial+\text{Im}\bar{\partial}$). Any such element satisfies
		\begin{equation}\label{zetaj}
			k_j+z_j\cdot x_i=\lambda_j\omega+\xi_j,
		\end{equation}
		where $\lambda_j\in\C$, $\xi_j\in\text{Im}\partial+\text{Im}\bar{\partial}$, and $\omega$ is a representative of the class generating $H^{n,n}_{A}(\Lambda V)$. The value of $\lambda_j$ is uniquely determined by $z_j$: indeed, if there exist elements as above such that
		\[
			k'_j+z_j\cdot x_i=\lambda'_j\omega+\xi'_j,\quad k_j+z_j\cdot x_i=\lambda_j\omega+\xi_j,
		\]
		subtracting the two expressions would yield $k'_j-k_j=(\lambda'_j-\lambda_j)\omega+(\xi'_j-\xi_j)$. The right-hand side is $\ddbar$-closed, while the left-hand side lies in $I_{i-1}(N)$. Thus, the cohomology class $(\lambda'_j-\lambda_j)[\omega]_A$ must vanish in $H^{n,n}_A(\Lambda V)$. As $[\omega]_A\neq0,$ this implies $\lambda'_j=\lambda_j$. Our goal is now to modify $x_i$ so that $\lambda_j=0$ for all such $z_j$.
		
		Since we have already addressed the case where $[\tau]_A=0$, we may restrict our attention on those $z_j$ such that $[z_j]_A\neq0$. We consider a collection of $\ddbar$-closed elements $\{z_j\}_j$ whose cohomology classes $\{[z_j]_{A}\}_j$ form a basis of $H^{n-p,n-q}_{A}(\Lambda V)$. The general result for an arbitrary $\tau$ then follows by linearity. 
		
		As $H^{n-p,n-q}_{A}(\Lambda V)$ is finite-dimensional, the non-degeneracy of the pairing defined by the $n$-SD condition ensures the existence of a Bott-Chern cohomology class $[\psi_i]_{BC}\in H^{p,q}_{BC}(\Lambda V)$ such that, for each $j$,
		\[
			[z_j]_{A}\cdot[\psi_i]_{BC}=\lambda_j[\omega]_A.
		\]
		In terms of representatives, this means that $z_j\cdot\psi_i=\lambda_j\omega+\theta_j$ for some $\theta_j\in\text{Im}\partial+\text{Im}\bar{\partial}$. We then define the modified generator as $\hat{x}_i:=x_i-\psi_i$. Notably, since $\psi_i$ is $d$-closed, it must lie in $\Lambda (V^{\le k-1}\oplus\ker\rho)\subseteq\mathcal{C}(\Lambda V_{i-1})$. 
		
		By \eqref{zetaj} and the definition of $\psi_i$, it follows immediately that any $\ddbar$-closed element of the form $k+z\cdot\hat{x}_i$ belongs to $\text{Im}\partial+\text{Im}\bar{\partial}$. In particular, this holds for any $\eta_0+\tau\cdot\hat{x}_i$. Thus, repeating the argument in Lemma \ref{technical1} with the new generator $\hat{x}_i$ concludes the proof.
	\end{proof}
\end{lemma}	

We are now in a position to prove an analogue of \cite[Theorem 3.1]{fernandezmunoz} in the bigraded setting. We provide a purely algebraic proof of the statement, which will be applied to the geometric setting in the following section.

\begin{theorem}\label{halfAlg}
	Let $\mathcal{M}=\Lambda V$ be a simply connected minimal $n$-SD cbba. Then $\mathcal{M}$ is strongly formal if and only if it is $(n-1)$-strongly formal.
	\begin{proof}
		In view of Proposition \ref{prop}, we only need to prove the converse implication. Our goal is to establish the $2n$-strong formality of $\mathcal{M}$ which, again by Proposition \ref{prop}, implies strong formality. We proceed by induction on the degree of the generators to construct the splitting $V^{p,q}=C^{p,q}\oplus N^{p,q}$ satisfying properties \ref{iii} and \ref{(iv)} of Definition \ref{sstrong}.
		
		Since $\mathcal{M}$ is $(n-1)$-strongly formal, Proposition \ref{ddbarlemma} ensures that the global $\ddbar$-Lemma holds. Then, it suffices to show that $(n+r-1)$-strong formality implies $(n+r)$-strong formality for each $r\ge 0$. Suppose $\mathcal{M}$ is $(n+r-1)$-strongly formal, meaning that the generators up to degree $n+r-1$ satisfy the required properties. Let $\rho$ be the map defined in \eqref{rho} for $k=n+r$. Let $\{x_1,...,x_p\}$ span $\ker\rho$ and complete them to a basis $\{x_1,...,x_{n_{n+r}}\}$ for $V^{n+r}$, ordered according to the minimality of the model: $dx_i\in\Lambda (V^{\le n+r-1}\oplus\langle x_1,..,x_{i-1}\rangle)\oplus V^{n+r+1}$. We can choose these generators to be of pure bidegrees (cf. \cite[Definition 2.3]{Stelzig_2025}). 
		
		Let $V_i:=V^{\le n+r-1}\oplus\langle x_1,...,x_i\rangle$. We prove by induction on $i\in\{1,...,n_{n+r}\}$ that each generator can be modified as $\hat{x}_i=x_i-\psi_i$, with $\psi_i\in \mathcal{C}(\Lambda \hat{V}_{i-1})$ and $\hat{V}_{i-1}:=V^{\le n+r-1}\oplus\langle\hat{x}_1,...,\hat{x}_{i-1}\rangle$ such that the space $\hat{V}_i$ admits a splitting $\hat{C}_i\oplus\hat{N}_i$ satisfying the conditions for $(n+r)$-strong formality. We note that such a modification of the generators preserves the subalgebra, i.e. $\mathcal{C}(\Lambda \hat{V}_i)=\mathcal{C}(\Lambda V_i)$.\\

		\noindent\textbf{Case 1: $x_i\in\ker\rho$.} By definition of $\rho$, there exists an element $\psi_i\in\mathcal{C}(\Lambda V^{\le n+r-1})$ such that $dx_i=d\psi_i$. We then define the modified generator $\hat{x}_i=x_i-\psi_i$, which is both $\partial$ and $\bar{\partial}$-closed. As established in the proof of Theorem \ref{characterization} for the element $b_y$, $\psi_i$ can be chosen of pure bidegree. We then set $\hat{C}_i=\hat{C}_{i-1}\oplus\langle\hat{x_i}\rangle$ and $\hat{N}_i=\hat{N}_{i-1}$. Let $I_i(\hat{N}):=(\hat{N}_{i}+\partial\hat{N}_{i}+\bar{\partial}\hat{N}_{i}+\ddbar\hat{N}_{i})\cdot\mathcal{C}(\Lambda \hat{V}_i)$ be the associated ideal. To verify condition \ref{iii} of Definition \ref{sstrong}, let $\eta\in I_i(\hat{N})$ be a $\partial$ and $\bar{\partial}$-closed element. We may thus write $\eta$ as
		\[
			\eta=\eta_0+\eta_1\cdot \hat{x}_i+...+\eta_k\cdot\hat{x}_i^k,
		\]
		where $\eta_j\in I_{i-1}(\hat{N})$ (note that in this case $I_i(\hat{N})=I_{i-1}(\hat{N})\cdot\Lambda(\hat{x}_i)$). Conditions $\partial\eta=0$ and $\bar{\partial}\eta=0$ imply $\partial\eta_j=0$ and $\bar{\partial}\eta_j=0$ for each $j=1,...,k$. By the inductive hypothesis, each $\eta_j$ is $\ddbar$-exact, which ensures $\eta\in\text{Im}\ddbar$.
		
		Similarly, to verify condition \ref{(iv)} of Definition \ref{sstrong}, let $\eta\in I_i(\hat{N})$ be $\ddbar$-closed. If we write $\eta$ as above, condition $\ddbar\eta=0$ implies $\ddbar\eta_j=0$ for each $j$. Again, by the inductive hypothesis, each $\eta_j$ belongs to $\text{Im}\partial+\text{Im}\bar{\partial}$, hence $\eta\in\text{Im}\partial+\text{Im}\bar{\partial}$.\\
		
		\noindent\textbf{Case 2: $x_i\notin\ker\rho$.} We set $\hat{C}_i=\hat{C}_{i-1}$, $N_i=\hat{N}_{i-1}\oplus\langle x_i\rangle$. We observe that if $x_i$ is modified by an element $\psi_i\in\mathcal{C}(\Lambda V^{\le n+r-1}\oplus\ker\rho)$, the differential $d$ remains injective on $\hat{x}_i=x_i-\psi_i$: indeed, $x_i\notin\ker\rho$ implies $dx_i\neq d\psi_i$.
		
		To establish property \ref{iii} of Definition \ref{sstrong}, it suffices to prove condition \ref{(iv)} for the ideal $I_i(\hat{N})$. Any $\partial$ and $\bar{\partial}$-closed element $\eta\in I_i(\hat{N})$ is, in particular, $\ddbar$-closed. If property \ref{(iv)} holds, then $\eta\in\text{Im}\partial+\text{Im}\bar{\partial}$. By the $\ddbar$-Lemma, we have $(\text{Im}\partial+\text{Im}\bar{\partial})\cap\ker\partial\cap\ker\bar{\partial}=\text{Im}\ddbar$ (cf. \cite[Lemma 5.15]{DeligneGriffiths1975}), which ensures $\eta\in\text{Im}\ddbar$.
		
		We now prove condition \ref{(iv)} of Definition \ref{sstrong}. This condition is trivially satisfied for elements with bidegree $(p,q)$ with $H_A^{p,q}(\Lambda V)=0$, in particular for $p>n$ or $q>n$ by the $n$-SD condition. Thus, we may focus on bidegrees with $p,q\le n$. Let $\eta\in I_i(N)$ be a $\ddbar$-closed element. If $|\eta|=(n,n)$, then by Lemmas \ref{technical1} and \ref{technical2}, condition \ref{(iv)} is satisfied up to considering a modified generator $\hat{x}_i=x_i+\psi_i$. We then assume $\eta\in I_i(\hat{N})$.
		
		Suppose $|\eta|<2n$. If $H_A^{|\eta|}(\Lambda V)=0$, $\eta$ is trivially in $\text{Im}\partial+\text{Im}\bar{\partial}$. Otherwise, the $n$-SD condition implies $H_{BC}^{(n,n)-|\eta|}(\Lambda V)\neq0$ by duality. Let $[\sigma]_{BC}\in H_{BC}^{(n,n)-|\eta|}(\Lambda V)$ be an arbitrary element and consider the product $[\sigma]_{BC}\cdot[\eta]_A=[\sigma\cdot\eta]_A$. 
		
		If $|\eta|>n$, then $|\sigma|<n$ and $\sigma\cdot\eta\in(I_{i}(\hat{N}))^{n,n}$ by degree reasons. 
		
		If $|\eta|=n$, then $|\sigma|=n$. Since the induction began at degree $n$, we can write $\sigma=\sigma_0+\sum_{j=1}^{n_n}c_jx_j$ with $\sigma_0\in\Lambda V^{<n}$ and $c_j\in\C$. Since $\sigma$ is $d$-closed, the only $x_j$ appearing in the linear combination are those spanning $\ker\rho$. Thus, $\sigma\in\Lambda V_{i-1}$, which again implies $\sigma\cdot\eta\in(I_{i}(\hat{N}))^{n,n}$.
		
		By the previous argument for bidegree $(n,n)$, the $\ddbar$-closed element $\sigma\cdot\eta$ lies in $\text{Im}\partial+\text{Im}\bar{\partial}$, hence $[\sigma\cdot\eta]_A=0$. Since this holds for any $[\sigma]_{BC}$, the non-degeneracy of the pairing implies $[\eta]_A=0$, which yields $\eta\in\text{Im}\partial+\text{Im}\bar{\partial}$.
	\end{proof}
\end{theorem}

Another application of Lemmas \ref{technical1} and \ref{technical2} is provided by the following result, which is geometrically motivated by the Lefschetz hyperplane Theorem (see \cite[Proposition 5.2.6]{Huybrechts2005}).

\begin{theorem}\label{hypThmAlg}
	Consider a morphism of cbba
	\[
		\rho:(B,\partial_B,\bar{\partial}_B)\to(A,\partial,\bar{\partial})
	\]
	such that the following properties hold:
	\begin{enumerate}
		\item \label{1hyp} $(B,\partial_B,\bar{\partial}_B)$ is a cohomologically simply connected, augmented cbba satisfying the $\ddbar$-Lemma, whose minimal model is a degree-wise finite-dimensional $(n-1)$-strongly formal cbba with $n\ge2$;
		\item \label{2hyp} $(A,\partial,\bar{\partial})$ is an augmented $n$-SD cbba satisfying the $\ddbar$-Lemma;
		\item \label{3hyp} $H^{\le n-1}_*(\rho)$ are isomorphisms and $H^{n}_*(\rho)$ is injective for $*\in\{BC, A\}$.
	\end{enumerate}
	Then $(A,\partial,\bar{\partial})$ is strongly formal.
	\begin{proof}
		To simplify the proof, we divide it into two steps: we first construct a model for $A$ starting from the model of $B$, and then we provide its $(n-1)$-strong formality.\\
		
		\noindent\textbf{Step 1: Construction of a minimal model for $A$.} Let $M_B=\Lambda V_B\to B$ be the $(n-1)$-strongly formal bigraded minimal model for $B$ and consider the composition $\psi:M_B\to B\overset{\rho}{\rightarrow}A$. It follows that $H_*^{\le n-1}(\psi)$ are isomorphisms and $H^n_*(\psi)$ is injective for $*\in\{BC, A\}$. 
		
		As a first step, we write $H^n(A)=\text{Im}H^n(\psi)\oplus H$ and define $M':=\Lambda(V_B\oplus H)$. We extend $\psi$ to a morphism $\psi':M'\to A$ by mapping each element of $H$ to one of its $\partial$ and $\bar{\partial}$-closed representatives. Since $A$ satisfies the $\ddbar$-Lemma, these representatives do not lie in $\text{Im}\partial+\text{Im}\bar{\partial}$. By construction, the maps $H^{\le n}(\psi')$ are isomorphisms. Moreover, $M'$ satisfies the $\ddbar$-Lemma, as the additional generators are closed and not in the image of $\partial$ or $\bar{\partial}$; in other words, we only added dots to $M_B$. Thus, by \cite[Theorem 5.17]{DeligneGriffiths1975}, $M'$ still satisfies the $\ddbar$-Lemma.
		
		Now, consider the cbba $\hat{M}:=\mathcal{C}(\Lambda(V_B^{\le n}\oplus H))$ and the restriction $\hat{\psi}:={\psi'}_{|_{\hat{M}}}$. The ideal $I_{n-1}'\subseteq M'$, defined as in Definition \ref{sstrong}, coincides with $I_{n-1}^B\subseteq M_B$ and does not contribute to the cohomology of $M'$. Furthermore, since the map $\hat{M}\to A$ factors through $B$ for generators of degree less than $n$, $I_{n-1}'$ does not contribute to the cohomology of $\hat{M}$.
		
		Our goal is to employ the algorithm in \cite[Theorem 2.34]{Stelzig_2025} to complete the partial model to $\hat{\psi}:\hat{M}\to A$. Before applying the algorithm, we augment $\hat{M}$ with additional generators and extend $\hat{\psi}$ to a map $\psi_R$ such that $H^{\le n}_{A}(\psi_R)$ and $H^{\le n}_{BC}(\psi_R)$ remain isomorphisms, and $H_{BC}^{n+1}(\psi_R)$ is injective. Let $K\subseteq(\hat{M})^{n+1}$ be a subspace mapping isomorphically onto $\ker(H^{n+1}_{BC}(\hat{\psi}))$; that is, $K$ is spanned by a set of representatives for $\ker(H^{n+1}_{BC}(\hat{\psi}))$. Note that $K$ cannot be generated by $\partial$ or $\bar{\partial}$-exact elements. Indeed, any $\partial$ or $\bar{\partial}$-exact generator of $K$ would be $\ddbar$-exact in $M'$ by the $\ddbar$-Lemma, in contradiction to $K$ defining non-trivial Bott-Chern cohomology classes. Moreover, $K$ is decomposable: the only generators in $(\hat{M})^{n+1}$ lie in $\text{Im}\partial+\text{Im}\bar{\partial}$, while $K\cap(\text{Im}\partial+\text{Im}\bar{\partial})=\emptyset$ as observed. 
		
		Define formal copies $R:=K[-1,-1]$, $\partial R:=K[0,-1]$, and $\bar{\partial}R:=K[-1,0]$. Using the notation $R=\langle r_k\,|\,k \text{ generator of } K\rangle$, we define the cbba
		\[
			M_R:=\hat{M}(R,\partial R,\bar{\partial}R\,|\,\ddbar r_k=k\quad\forall k\text{ generator of } K).
		\]
		Since $K$ consists of decomposable elements, $M_R$ remains a minimal cbba. By extending $\hat{\psi}$ to a map $\psi_R$ that sends each $r_k$ to a $\ddbar$-primitive of $\hat{\psi}(k)$, we obtain a cbba morphism $\psi_R:M_R\to A$ such that $H_A^{\le n}(\psi_R)$ and $H_{BC}^{\le n}(\psi_R)$ are isomorphisms, and $H^{n+1}_{BC}(\psi_R)$ is injective. Accordingly, we set $N_R^{\le n-2}=N_B^{\le n-2}$, $N_R^{n-1}=N_B^{n-1}\oplus R$.
		
		We now apply the algorithm developed in \cite[Theorem 2.34]{Stelzig_2025}. Since surjectivity in $H_A^n(\psi_R)$ is already satisfied by construction, the algorithm proceeds by establishing injectivity in $H_{BC}^{n+2}$. To this end, it supplements the partial model $M_R$ with additional generators in degree $n$ and $n+1$ to reduce the Bott-Chern cohomology in degree $n+2$. Notably, these generators do not modify the ideal $I_{n-1}^R$. This remains true in the subsequent steps of the algorithm: to ensure surjectivity in $H_A^{\ge n+1}$ and injectivity in $H_{BC}^{\ge n+3}$, we only need to add generators in degrees $\ge n+1$; hence, $I_{n-1}^R$ remains invariant throughout the process.
		
		This construction yields to a holomorphically simply connected minimal model $\varphi:M_A\to A$ such that $\varphi_{|_{M_R}}=\psi_R$ and $I_{n-1}^A=I_{n-1}^R$, satisfying the $\ddbar$-Lemma. Moreover, the ideal $I_{n-1}^B\subseteq I_{n-1}^A$ does not contribute to the cohomology of $M_A$, as $\varphi$ still factors through $B$ on generators of degree $\le n-1$.\\
		
		\noindent\textbf{Step 2: The model $M_A$ is $(n-1)$-strongly formal.} We show that the argument from the proof of Theorem \ref{halfAlg} can be applied to elements in $I_{n-1}^A$, allowing us to conclude that this ideal does not contribute to the cohomology of $M_A$. Since $n\ge2$, $A$ remains connected and holomorphically simply connected by hypothesis \ref{3hyp}. We proceed by induction on the generators of $R$. Specifically, we assume that the ideal $I_{j-1}$, which incorporates the generators $r_1,...,r_{j-1}$, does not contribute to the cohomology of the model. We then show that $\ddbar$-closed elements in the ideal $I_j$, obtained by incorporating $r_j$ among its generators, can be written as in \eqref{etaClaim1}; this allows us to proceed as in the proof of Theorem \ref{halfAlg}. The base case is satisfied since $I_0=I_{n-1}^B$.
		
		Let $N_{j-1}:=N_B^{\le n-1}\oplus\langle r_1,...,r_{j-1}\rangle$ and $V_{j-1}:=V_B^{\le n-1}\oplus\langle r_1,...,r_{j-1}\rangle$. The ideal $I_j$ can be written as in the hypotheses of Lemma \ref{technical1}. As noted above, condition \ref{iii} of Definition \ref{sstrong} is implied by condition \ref{(iv)} and the $\ddbar$-Lemma; thus, it suffices to prove that any $\ddbar$-closed element in $I_j$ lies in $\text{Im}\partial+\text{Im}\bar{\partial}$. 
		
		Consider an arbitrary $\ddbar$-closed element $\eta\in I_j$. If $|\eta|\ge2n$, the $n$-SD condition together with Lemmas \ref{technical1} and \ref{technical2} imply $\eta\in\text{Im}\partial+\text{Im}\bar{\partial}$. Following the argument in Theorem \ref{halfAlg}, we analyse the cases $|\eta|=n-1$ and $|\eta|=n$ separately to reduce the proof to the $(n,n)$-degree case via multiplication by a generic element of $H^{(n,n)-|\eta|}_{BC}(M_A)$. 
		
		If $|\eta|=n-1$, then $\eta=\eta_0+c\cdot r_j$ with $\eta_0\in I_{j-1}$ and $c\in\C$. Here, $\eta$ is $\ddbar$-closed if and only if $c=0$ (since $\ddbar r_j\notin\ddbar I_{j-1}$ by definition of $r_j$); thus $\eta=\eta_0\in I_{j-1}$, which implies $[\eta]_A=[\eta_0]_A=0$.
		
		If $|\eta|=n$, since $M_A$ lacks $1$-degree generators, we have $\eta=\eta_0+a\partial r_j+b\bar{\partial}r_j$ with $\eta_0\in I_{j-1}$ and $a,b\in\C$. It follows that $[\eta]_A=[\eta_0]_A=0$.
		
		The case $|\eta|=2n-1$ also requires separate considerations, as linear combinations of $r_j\wedge\partial r_j$ and $r_j\wedge\bar{\partial}r_j$ could potentially arise (unlike the cases $|\eta|=2n$ or $|\eta|<2n-1$). However, $H^{2n-1}(A)=0$ due to the simply connectedness of $A$.
		
		By an argument analogous to the one used in the proof of Theorem \ref{halfAlg}, we conclude that $I_j$ does not contribute to the cohomology of $M_A$ for any $j$. Consequently, $M_A$ is $(n-1)$-strongly formal and, by Theorem \ref{halfAlg}, it is strongly formal.
	\end{proof}
\end{theorem}

\section{Applications in complex geometry}\label{section3}

In this section, we apply the results established in the previous section to prove the strong formality of several classes of complex manifolds. Since we consider $(n-1)$-strongly formal $n$-dimensional complex manifolds, by Proposition \ref{ddbarlemma}, without loss of generality, we may assume the $\ddbar$-Lemma to hold. We observe that minimal models of $n$-dimensional, compact, connected manifolds are in particular $n$-SD minimal cbbas. Thus, Theorem \ref{halfAlg} immediately yields the following result.

\begin{theorem}\label{half}
	Let $X$ be a compact, connected $\ddbar$-manifold with $H_{dR}^1(X)=0$ and $\dim_{\C}(X)=n$. Then $M$ is strongly formal if and only if $X$ is $(n-1)$-strongly formal.
\end{theorem}

The triviality of the first de Rham cohomology group of the manifold and the $\ddbar$-Lemma imply its holomorphically simply connectedness, hence the existence of a minimal model in this hypotheses is guaranteed by Theorem \ref{thmexistence}. Theorem \ref{half} can be employed to establish the strong formality of various classes of compact complex manifolds. We first consider compact K\"ahler manifolds with central cohomology of width $l$, i.e., K\"ahler manifolds $(X,\omega)$ such that, for $k<\dim_{\C} X-l$, $H^k(X,\C)=0$ if $k$ is odd, and $H^k(X,\C)=\C\langle[\omega^{\frac{k}{2}}]\rangle$ if $k$ is even. Note that for $l=0$, we include the case of complete intersections (cf. \cite[\S3.2]{Stelzig_2025}).

\begin{theorem}\label{centralcohomology}
	Any compact K\"ahler manifold $(X,\omega)$ with $\dim_{\C} X=n\ge2$ and central cohomology of width $\frac{n}{2}-1$ is strongly formal. For $n=4m-1$, the same result holds also if $h^{m,m}(X)=2$.
	\begin{remark}
		The width $\frac{n}{2}-1$ is optimal: the example constructed in \cite{placini2024nontrivialmasseyproductscompact} is a compact K\"ahler manifold that is not strongly formal; this manifold is simply connected (and hence possesses central cohomology of width $\frac{n}{2}$), but it lacks central cohomology of width $\frac{n}{2}-1$. 
	\end{remark}
	\begin{proof}
		The proof consists in the construction of an $(n-1)$-strongly formal minimal model for the manifold, ensuring its strong formality by Theorem \ref{half}. We first prove the statement for manifolds with central cohomology. We begin by constructing a partial bigraded model for $A_X$ and then employ the algorithm in \cite[Theorem 2.34]{Stelzig_2025} to complete it to a minimal model. Since the width is required to be $\frac{n}{2}-1$, excluding the powers of the K\"ahler class, there are no non-trivial cohomology classes in degree less than $\frac{n}{2}+1$.
		
		Let $x$ be an element of bidegree $(1,1)$ and consider the cbba $\mathcal{M}_1:=\Lambda(x)$, equipped with the morphism $\varphi_1:\mathcal{M}_1\to A_X$ defined by $\varphi_1(x)=\omega$. We denote by $P_k$ a set of representatives for the primitive cohomology $P^k(X,\C)$ in degree $k$. Recall that $P^k(X,\C)$ consist of cohomology classes $[\alpha]$ such that $\Lambda[\alpha]=0$, where $\Lambda$ is the adjoint of the Lefschetz operator $L:[\alpha]\to[\omega]\wedge[\alpha]$. By the Hard Lefschetz decomposition, we have 
		\[
			H^k(X,\C)=\bigoplus_{j\ge0}L^j(P^{k-2j}(X,\C))\cong\bigoplus_{j\ge0}x^j\wedge P_{k-2j}.
		\]
		Since $L^{n-k}:H^k(X,\C)\to H^{2n-k}(X,\C)$ is an isomorphism, there are no primitive forms in degrees greater than $n$. We then augment $\mathcal{M}_1$ with all the primitive representatives $P_k$ for $\frac{n}{2}+1\le k\le n$, thereby obtaining the minimal cbba $\mathcal{M}_2:=\Lambda(\mathcal{M}_1\oplus\bigoplus_kP_k)$. We define $\varphi_2:\mathcal{M}_2\to A_X$ by extending $\varphi_1$ such that $\varphi_2(p)=p$ for each $p\in P_k$ in the range $\frac{n}{2}+1\le k\le n$. 
		
		Due to the width limitation, the product of any two primitives has degree at least $n+2$. Similarly, by the Hard Lefschetz decomposition, the first trivial product between a power of the K\"ahler class and a primitive class occurs in degree greater or equal to $n+2$ (specifically, $[\omega]\wedge P^n(X,\C)=0$). Thus, $\varphi_2$ induces isomorphisms in cohomology up to degree $n+1$. 
		
		We employ the algorithm from \cite[Theorem 2.34]{Stelzig_2025} to complete the partial model $(\mathcal{M}_2,\varphi_2)$ to a minimal model $(\mathcal{M},\varphi)$, starting from degree $n+1$. Note that this algorithm proceeds inductively, building the model in degree $k$ by only adding generators in degrees $k$ and $k-1$. According to the algorithm, the model is fixed in degree $k$ whenever we have injectivity in $H_{BC}^{\le k+1}(\varphi)$ and surjectivity in $H_A^{\le k-1}(\varphi)$. Consequently, our partial model $(\mathcal{M}_2,\varphi_2)$ is modified only starting from degree $n$.
		
		Once the complete model $(\mathcal{M},\varphi)$ is obtained, we consider the generators in degrees less then or equal to $n-1$, specifically $V^2=\C\langle x\rangle$ and $V^k=P_k$ for $\frac{n}{2}+1\le k\le n-1$. The differential vanishes on these generators, $d(V^{\le n-1})=0$, by construction. We then define $C^{p,q}:=V^{p,q}$, $N^{p,q}:=0$ for $p+q\le n-1$. Conditions \ref{iii} and \ref{(iv)} of Definition \ref{sstrong} are trivially satisfied since the ideal $I_{n-1}$ is zero, while the global $\ddbar$-Lemma holds due to the K\"ahler condition. It follows that the manifold is $(n-1)$-strongly formal and hence, by Theorem \ref{half}, it is strongly formal.
		
		Suppose now that $n=4m-1$ and, in addition to the central cohomology assumption, we have $h^{m,m}(X)=2$. We begin by taking $\Lambda(x)$ and adding the primitive representatives. In degree $2m$, the generators are given by $V^{2m}=\C\langle x^{m}, \eta\rangle$. By the Hard Lefschetz decomposition, we necessarily have
		\[
			[\eta^2]_{BC}\in \bigoplus_{j>0}[\omega^j]_{BC}\wedge P^{4m-2j}(X,\C),\quad\text{then}\quad[\eta^2]_{BC}=\sum_{j>0}[\omega^j\wedge\alpha_{4m-2j}]_{BC},
		\]
		where $\alpha_{4m-2j}\in P_{4m-2j}$. We augment the partial model with a generator $\xi$ of degree $4m-2=n-1$ subject to the condition $\ddbar\xi=\eta^2-\sum_{j>0}x^j\wedge\alpha_{4m-2j}$. We then define $N^{n-1}=\C\langle\xi\rangle$ and observe that the ideal $I_{n-1}=\C\langle\xi,\partial\xi,\bar{\partial}\xi,\ddbar\xi\rangle\cdot\mathcal{C}(\Lambda V^{\le n-1})$ does not contribute to the cohomology of the entire model. Indeed, an arbitrary element in $I_{n-1}$ can be written as
		\begin{align*}
			\gamma=&\xi\wedge(a_0+\xi\wedge a_1+\partial\xi\wedge a_2+\bar{\partial}\xi\wedge a_3)+\partial\xi\wedge(b_0+\bar{\partial}\xi\wedge b_1)\\
			&+\bar{\partial}\xi\wedge c_0+\ddbar\xi\wedge d_0,
		\end{align*}
		where $a_i,b_i,c_0,d_0\in\Lambda C^{\le n-1}$. Imposing $\partial\gamma=0$ and $\bar{\partial}\gamma=0$ yields
		\[
			\gamma=\frac{1}{2}\ddbar(\xi^2)\wedge b_1+\ddbar\xi\wedge d_0
			=\ddbar\left(\frac{1}{2}\xi^2\wedge b_1+\xi\wedge d_0\right),
		\]
		which implies that $I_{n-1}$ does not contribute to the Bott-Chern cohomology of the entire model. Similarly, if we require $\ddbar\gamma=0$, we obtain
		\begin{align*}
			\gamma=&\frac{1}{2}\ddbar(\xi^2)\wedge b_1+\frac{1}{2}\partial(\xi^2)\wedge a_2+\frac{1}{2}\bar{\partial}(\xi^2)\wedge a_3+\partial\xi\wedge b_0+\bar{\partial}\xi\wedge c_0+\ddbar\xi\wedge d_0\\
			=&\partial\left(\frac{1}{2}\xi^2\wedge a_2+\xi\wedge b_0\right)+\bar{\partial}\left(\frac{1}{2}\xi^2\wedge a_3+\xi\wedge c_0\right)+\ddbar\left(\frac{1}{2}\xi^2\wedge b_1+\xi\wedge d_0\right),
		\end{align*}
		which implies that $I_{n-1}$ does not contribute to the Aeppli cohomology of the entire model. As the properties in Definition \ref{sstrong} are fully satisfied, it follows that the partial model is $(n-1)$-strongly formal. Finally, applying the algorithm starting from degree $n+1$ yields the strong formality of $X$.
	\end{proof}
\end{theorem}

As a consequence of this result, we have established the strong formality of certain Fano manifolds.

\begin{corollary}
	Fano threefolds with Picard group of rank $1$ or $2$ and Fano fourfolds with Picard group of rank $1$ are strongly formal.
\end{corollary}

This reasoning can furthermore be applied to other classes of not necessarily K\"ahler manifolds. Consider a compact, connected and holomorphically simply connected $n$-dimensional $\ddbar$-manifold $X$ with no non-trivial multiplicative relations in cohomology below degree $n+2$, in the sense that the multiplication map
\begin{align*}
	\mathrm{Sym}^k(H^{\le k}_{BC}(X))&\to H^k_{BC}(X)\\
	[\alpha]_{BC}\otimes[\beta]_{BC}&\mapsto[\alpha\wedge\beta]_{BC}
\end{align*}
is injective for all $k<n+2$. Here, $\mathrm{Sym}(H^{\le k}_{BC}(X))$ denotes the graded symmetric algebra generated by $H_{BC}^{\le k}(X)$. One can construct a partial model following the previous proof and complete it to a minimal model applying the algorithm in \cite[Theorem 2.34]{Stelzig_2025}. Since the algorithm does not affect elements in degree $\le n-1$ and since the partial model is trivially $(n-1)$-strongly formal, the same holds for the complete model. Consequently, if the manifold is compact, connected and holomorphically simply connected, Theorem \ref{half} ensures its strong formality. These observations are summarized in the following result.

\begin{theorem}\label{generalthm}
	A compact, connected and holomorphically simply connected $\ddbar$-manifold of complex dimension $n$ with no multiplicative relations in cohomology below degree $n+2$ is strongly formal.
\end{theorem}

As a direct application of this result, we establish the strong formality of $\ddbar$-complex structures on $\mathbb{S}^3\times\mathbb{S}^3$. On this class of manifolds, multiplicative relations in cohomology only occur in degree $6$, meaning that the hypotheses of Theorem \ref{generalthm} are satisfied. In particular, we find in this class Clemens manifolds with vanishing second Betti number (see \cite{Friedman91} for an explicit construction). These simply connected $3$-dimensional complex manifolds are compact and not necessary K\"ahler; nevertheless, they satisfy the $\ddbar$-Lemma, as shown in \cite[Theorem 1.1]{Chi24}. 

\begin{corollary}
	Clemens manifolds with vanishing second Betti number are strongly formal.
\end{corollary}

The next class of K\"ahler manifolds under consideration arises from the Lefschetz hyperplane Theorem. We shall use this result to establish the strong formality of certain hypersurfaces of an $(n-1)$-strongly formal K\"ahler manifold.

\begin{theorem}[see e.g. \protect{\cite[Proposition 5.2.6]{Huybrechts2005}}]\label{hyperplane}
	Let $X$ be a compact K\"ahler manifold of complex dimension $n+1$, and let $Y\subseteq X$ be a smooth hypersurface such that the induced line bundle $\mathcal{O}(Y)$ is positive. Then the map induced in cohomology by the canonical restriction map 
	\[
		H^k(res):H^k(X;\C)\to H^k(Y;\C)
	\]
	is an isomorphism for $k\le n-1$ and injective for $k=n$.
\end{theorem}

Let $X$ and $Y$ be compact K\"ahler manifolds as in Theorem \ref{hyperplane}. If $X$ is connected, holomorphically simply connected and $(n-1)$-strongly formal, the hypotheses of Theorem \ref{hypThmAlg} are satisfied. Consequently, we obtain the following result.

\begin{theorem}\label{hypThm}
	Let $X$ be a compact, connected K\"ahler manifold of complex dimension $(n+1)\ge3$ with $H_{dR}^1(X)=0$. Let $Y\subseteq X$ be a smooth hypersurface such that the induced line bundle $\mathcal{O}(Y)$ is positive. If $X$ is $(n-1)$-strongly formal, then $Y$ is strongly formal.
\end{theorem}

We remark that the assumption on the dimension of $X$ cannot be removed; indeed, if one considers a generic homogeneous polynomial of degree $d$ in three variables, its zero locus in $\CP^2$ is a Riemann surface of genus $\frac{(d-1)(d-2)}{2}$. In particular, for $d>3$ this construction results in a Riemann surface of genus $>1$, which is never strongly formal, as established in \cite[Theorem A]{placini2024nontrivialmasseyproductscompact}. Moreover, if $(n+1)\ge3$ and $X$ satisfies $H_{dR}^1(X)=0$, Theorem \ref{hyperplane} implies $H_{dR}^1(Y)=0$.

As a corollary of Theorem \ref{hypThm}, we establish the strong formality of generalized complete intersection manifolds with trivial first de Rham cohomology group. Consider a product of complex projective spaces $M=\CP_1^{N_1}\times...\times\CP_k^{N_k}$. A smooth manifold $X\subseteq M$ of complex codimension $m$ is a generalized complete intersection if it is the smooth intersection of the zero loci of $m$ sections of line bundles of the form $\mathcal{O}_1(d_1)\otimes...\otimes\mathcal{O}_k(d_k)$, where $\mathcal{O}_j(d_j):=\mathcal{O}_{\CP_j^{N_j}}(d_j)$.

\begin{corollary}
	Let $X$ be a generalized complete intersection manifold defined by positive line bundles with $H_{dR}^1(X)=0$. Then $X$ is strongly formal.
	\begin{proof}
		Let $X$ be defined by sections in $M=\CP_1^{N_1}\times...\times\CP_k^{N_k}$. Since $M$ is strongly formal by \cite[Proposition 4.21]{milivojevic2024bigradednotionsformalityaepplibottchernmassey}, the smooth zero locus of a positive line bundle on $M$ remains strongly formal by Theorem \ref{hypThm}. Consequently, as long as a generalized complete intersection has trivial first de Rham cohomology group, we can inductively apply Theorem \ref{hypThm} to ensure its strong formality.
	\end{proof}
\end{corollary}

\bibliography{sStrongFormality}

\begin{bibdiv}
\begin{biblist}

\bib{AngTar16}{article}{
      author={Angella, D.},
      author={Tardini, N.},
       title={Quantitative and qualitative cohomological properties for
  non-{K}\"ahler manifolds},
        date={2017},
     journal={Proc. Amer. Math. Soc.},
      volume={145},
      number={1},
       pages={273\ndash 285},
         url={https://gecogedi.dimai.unifi.it/paper/1/},
}

\bib{Angella_2015}{article}{
      author={Angella, D.},
      author={Tomassini, A.},
       title={On {B}ott-{C}hern cohomology and formality},
        date={2015},
        ISSN={0393-0440},
     journal={Journal of Geometry and Physics},
      volume={93},
       pages={52\ndash 61},
         url={http://dx.doi.org/10.1016/j.geomphys.2015.03.004},
}

\bib{cesarino2025aepplibottchernmasseyproductsnonkahler}{misc}{
      author={Cesarino, N.},
      author={Tomassini, A.},
       title={Aeppli-{B}ott-{C}hern {M}assey products on non-{K}\"ahler
  solvmanifolds},
        date={2025},
         url={https://arxiv.org/abs/2512.05894},
        note={\url{https://arxiv.org/abs/2512.05894}},
}

\bib{Chi24}{article}{
      author={Chi, L.},
       title={Polarized {H}odge {S}tructures for {C}lemens {M}anifolds},
        date={2024},
     journal={Mathematische Annalen},
      volume={389},
       pages={525\ndash 541},
}

\bib{DeligneGriffiths1975}{article}{
      author={Deligne, P.},
      author={Griffiths, P.},
      author={Morgan, J.},
      author={Sullivan, D.},
       title={Real homotopy theory of {K}\"ahler manifolds},
        date={1975},
        ISSN={1432-1297},
     journal={Inventiones mathematicae},
      volume={29},
      number={3},
       pages={245\ndash 274},
}

\bib{fernandezmunoz}{article}{
      author={Fernandez, M.},
      author={Muñoz, V.},
       title={Formality of {D}onaldson submanifolds},
        date={2005},
     journal={Math Z.},
      volume={500},
       pages={149\ndash 175},
}

\bib{Friedman91}{article}{
      author={Friedman, R.},
       title={On threefolds with trivial canonical bundle},
        date={1991},
     journal={Proceedings of Symposia in Pure Mathematics},
       pages={103\ndash 134},
}

\bib{Huybrechts2005}{book}{
      author={Huybrechts, D.},
       title={Complex {G}eometry},
    subtitle={An {I}ntroduction},
      series={Universitext},
   publisher={Springer Berlin Heidelberg},
        date={2005},
}

\bib{martínmerchán2025holomorphiclinkingnumbersabc}{misc}{
      author={Martín-Merchán, L.},
      author={Stelzig, J.},
       title={Holomorphic linking numbers, {A}{B}{C} {M}assey products, and
  {C}alabi-{Y}au 3-folds},
        date={2025},
         url={https://arxiv.org/abs/2512.02187},
        note={\url{https://arxiv.org/abs/2512.02187}},
}

\bib{milivojevic2024bigradednotionsformalityaepplibottchernmassey}{article}{
      author={Milivojevic, A.},
      author={Stelzig, J.},
       title={Bigraded notions of formality and
  {A}eppli-{B}ott-{C}hern-{M}assey products},
        date={2024},
     journal={Communications in Analysis and Geometry},
      volume={32},
       pages={2901\ndash 2933},
}

\bib{Neisendorfer78}{article}{
      author={Neisendorfer, J.},
      author={Taylor, L.},
       title={Dolbeault {H}omotopy {T}heory},
        date={1978},
        ISSN={00029947},
     journal={Transactions of the American Mathematical Society},
      volume={245},
       pages={183\ndash 210},
         url={http://www.jstor.org/stable/1998862},
}

\bib{placini2024nontrivialmasseyproductscompact}{misc}{
      author={Placini, G.},
      author={Stelzig, J.},
      author={Zoller, L.},
       title={Nontrivial {M}assey products on compact {K}\"ahler manifolds},
        date={2024},
         url={https://arxiv.org/abs/2404.09867},
        note={\url{https://arxiv.org/abs/2404.09867}},
}

\bib{placini2025strongformalitytorichomogeneous}{misc}{
      author={Placini, G.},
      author={Stelzig, J.},
      author={Zoller, L.},
       title={Strong formality of toric and homogeneous compact {K}\"ahler
  manifolds},
        date={2025},
         url={https://arxiv.org/abs/2510.17288},
        note={\url{https://arxiv.org/abs/2510.17288}},
}

\bib{RUBINI2025105586}{article}{
      author={Rubini, L.},
       title={Some computations on trivial canonical-bundle solvmanifolds},
        date={2025},
        ISSN={0393-0440},
     journal={Journal of Geometry and Physics},
      volume={216},
       pages={105586},
}

\bib{SFERRUZZA2022104470}{article}{
      author={Sferruzza, T.},
      author={Tomassini, A.},
       title={Dolbeault and {B}ott-{C}hern formalities: {D}eformations and
  $\partial\bar{\partial}$-lemma},
        date={2022},
        ISSN={0393-0440},
     journal={Journal of Geometry and Physics},
      volume={175},
       pages={104470},
  url={https://www.sciencedirect.com/science/article/pii/S0393044022000201},
}

\bib{sferruzza2024bottchernformalitymasseyproducts}{article}{
      author={Sferruzza, T.},
      author={Tomassini, A.},
       title={Bott-{C}hern formality and {M}assey products on strong {K}\"ahler
  with torsion and {K}\"ahler solvmanifolds},
        date={2024},
     journal={J. Geom. Anal.},
      volume={34},
      number={348},
}

\bib{sferruzza2025hermitiangeometricallyformalmanifolds}{misc}{
      author={Sferruzza, T.},
      author={Tomassini, A.},
       title={Hermitian geometrically formal manifolds},
        date={2025},
         url={https://arxiv.org/abs/2507.10263},
        note={\url{https://arxiv.org/abs/2507.10263}},
}

\bib{Stelzig_2021}{article}{
      author={Stelzig, J.},
       title={On the structure of double complexes},
        date={2021},
        ISSN={1469-7750},
     journal={Journal of the London Mathematical Society},
      volume={104},
      number={2},
       pages={956\ndash 988},
         url={http://dx.doi.org/10.1112/jlms.12453},
}

\bib{Stelzig_2025}{article}{
      author={Stelzig, J.},
       title={Pluripotential homotopy theory},
        date={2025},
        ISSN={0001-8708},
     journal={Advances in Mathematics},
      volume={460},
       pages={110038},
         url={http://dx.doi.org/10.1016/j.aim.2024.110038},
}

\bib{Sullivan75}{incollection}{
      author={Sullivan, D.},
       title={Differential forms and the topology of manifolds},
        date={1975},
   booktitle={Manifolds---{T}okyo 1973 ({P}roc. {I}nternat. {C}onf., {T}okyo,
  1973)},
   publisher={Math. Soc. Japan, Tokyo},
       pages={37\ndash 49},
}

\bib{tardini2015geometricbottchernformalitydeformations}{article}{
      author={Tardini, N.},
      author={Tomassini, A.},
       title={On geometric {B}ott-{C}hern formality and deformations},
        date={2017},
        ISSN={0373-3114,1618-1891},
     journal={Ann. Mat. Pura Appl. (4)},
      volume={196},
      number={1},
       pages={349\ndash 362},
}

\end{biblist}
\end{bibdiv}

\end{document}